%% file: paper.tex
\documentclass[10pt, a4paper, fleqn, final]{article}
\usepackage[utf8x]{inputenc}
\usepackage[T1]{fontenc}
\usepackage[english]{babel}
\usepackage{amsmath, amsthm, amssymb, mathtools}
\usepackage{libertine}
\usepackage{graphs}
\usepackage{hyperref, url, breakurl}
\usepackage{graphicx}
\usepackage{xcolor}

\newtheorem{theo}{Theorem}
\newtheorem{coro}[theo]{Corollary}
\newtheorem{lemm}[theo]{Lemma}

\newtheorem{ques}{Question}
\theoremstyle{definition}
\newtheorem{defi}[theo]{Definition}
\newtheorem{rema}[theo]{Remark}

\newcommand{\eqspace}{\ensuremath{\mathrel{\phantom{=}}}}
\newcommand{\cdotcup}{\mathrel{\mathaccent\cdot\cup}}
\newcommand{\esg}[2]{#1\langle#2\rangle} % notation for edge subgraphs (spanning subgraphs)
\newcommand{\eig}[2]{#1[#2]} % notation for edge induced subgraphs
\DeclarePairedDelimiter\abs{\lvert}{\rvert}
\newcommand{\pcoef}[2]{[#1](#2)} % notation for "coefficient" in front of #1 of a polynomial #2
\newcommand{\pdeg}[2]{\deg_{#1}(#2)} % notation for degree of #1 of polynomial #2

\title{The covered components polynomial: \\ A new representation of \\ the edge elimination polynomial}
\author{Martin Trinks\thanks{trinks@hs-mittweida.de, Hochschule Mittweida, University of Applied Sciences, Faculty Mathematics / Sciences / Computer Science, Technikumplatz 17, 09648 Mittweida, Germany}}
\date{\today}

\begin{document}

\maketitle

{\center\setlength{\parindent}{0pt}
\begin{minipage}[c]{5.7cm}
The author receives the grant 080940498 from the European Social Fund (ESF) of the European Union (EU).
\end{minipage}
\hspace{0.4cm}
\begin{minipage}[c]{5.7cm}
\centering
\includegraphics[width=3cm]{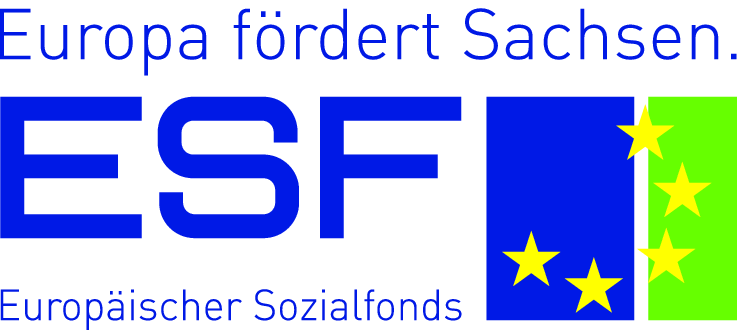}
\hspace{0.25cm}
\includegraphics[width=2cm]{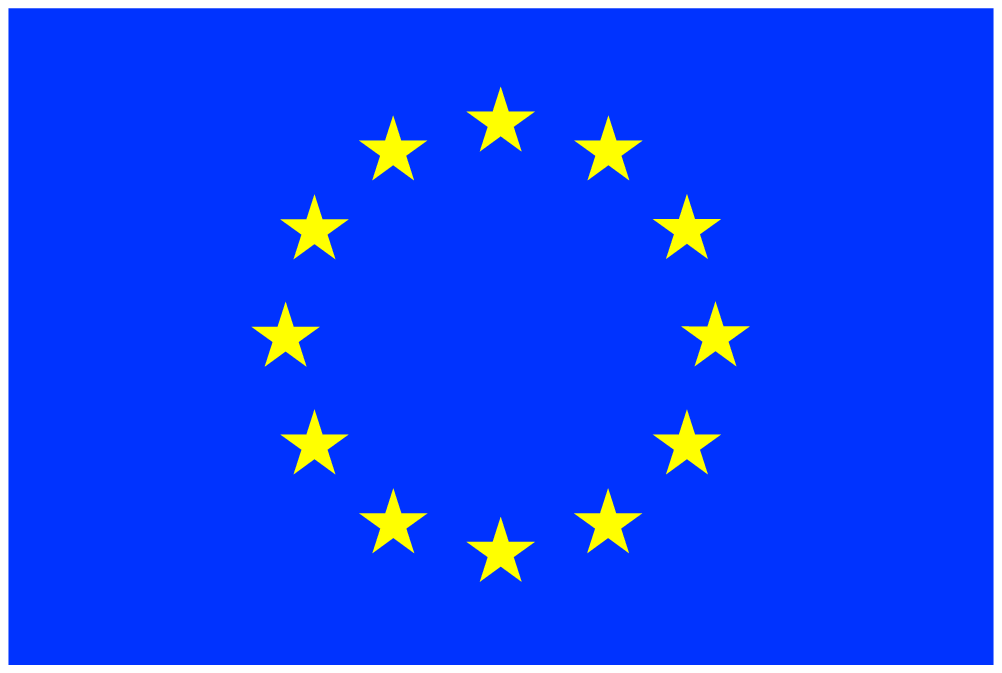}
\end{minipage}
}

\begin{abstract}
Motivated by the definition of the edge elimination polynomial of a graph we define the covered components polynomial counting spanning subgraphs with respect to their number of components, edges and covered components. We prove a recurrence relation, which shows that both graph polynomials are substitution instances of each other. We give some properties of the covered components polynomial and some results concerning relations to other graph polynomials.
\end{abstract}

\input{introduction}
\input{definition}
\input{invariants}
\input{unique_graphs}
\input{special_graphs}
\input{relations}

\input{random_subgraphs}
\input{specialization_forests}
\input{generalization_hypergraphs}
\input{conclusions}

\end{document}

%% file: introduction.tex
\section{Introduction}

\label{sect:introduction}

Averbouch, Godlin and Makowsky \cite{averbouch2008, averbouch2010} define the \emph{edge elimination polynomial} of a graph by a recurrence relation with respect to the deletion, contraction and elimination of an edge. This is motivated by several well-studied graph polynomials satisfying such recurrence relations and the long-term objective of a ``general theory of graph polynomials'' \cite{makowsky2008}. The authors prove that the edge elimination polynomial is the most general graph polynomial satisfying such a recurrence relation.

We define the \emph{covered components polynomial} of a graph counting spanning subgraphs with respect to their number of components, edges and covered components. This again is motivated by the usage of the number of covered components, which is the number of connected components not being a single vertex, in an expansion of the edge elimination polynomial \cite[Theorem 5]{averbouch2008}. We prove that the covered components polynomial is just another representation of the edge elimination polynomial as both graph polynomials can be calculated from each other. Furthermore, we determine a multitude of properties of both graph polynomials and their relations to other graph polynomials. We show that for forests both polynomials can be calculated from the \emph{bivariate chromatic polynomial} defined by Dohmen, Pönitz and Tittmann \cite{dohmen2003}.

The paper is organized as follows. In the remainder of this section we give necessary notations and introduce the edge elimination polynomial. In Section \ref{sect:definition} the covered components polynomial is defined and the relation to the edge elimination polynomial is proven. Some graph invariants coded in the covered components polynomial are given in Section \ref{sect:invariants}. These graph invariants are used in Section \ref{sect:unique_graphs} to determine some graph classes that are unique with respect to this graph polynomial. In Section \ref{sect:special_graphs} recursive and closed forms for the covered components polynomial of some graph classes are given. An overview about the relations from the covered components polynomial to other graph polynomials and some results about the distinctive power are the topic of Section \ref{sect:relations}. Section \ref{sect:random_subgraphs} presents an example in case of edges being intact/appearing with a given probability. Some results special to forest and the generalization to hypergraphs can be found in Section \ref{sect:specialization_forests} and Section \ref{sect:generalization_hypergraphs}, respectively. The paper ends with some open problems in Section \ref{sect:open_problems}.

For the sake of convenience, we do not state the following results for both the covered components polynomial and the edge elimination polynomial. Each statement for one of them also hold (at least analogously) for the other one.

A graph \(G = (V, E)\) is an ordered pair of a set \(V\), the vertex set, and a multiset \(E\), the edge set, such that the elements of the edge set are one- and two-element subsets of the vertex set, \(e \in \binom{V}{1} \cup \binom{V}{2}\) for all \(e \in E\). For an edge subset \(A \subseteq E\) of \(G\), \(\esg{G}{A} = (V, A)\) is the spanning subgraph of \(G\) spanned by \(A\). By \(k(G)\) and \(c(G)\) we denote the number of connected components, component for short, and the number of covered connected components, covered component for short. A covered component is a component including at least one edge. 

For a graph \(G = (V, E)\) with an edge \(e \in E\) of \(G\) we define the following three edge operations:
\begin{itemize}
\item \(-e\): deletion of the edge \(e\), i.e., edge \(e\) is removed,
\item \(/e\): contraction of the edge \(e\), i.e., edge \(e\) is removed and incident vertices are merged (parallel edges and loops may occur),
\item \(\dagger e\): extraction of the edge \(e\), i.e., edge \(e\) and its incident vertices are removed.
\end{itemize}
The resulting graphs are denoted by \(G_{-e}\), \(G_{/e}\) and \(G_{\dagger e}\), respectively. For two graphs \(G^1\) and \(G^2\), \(G^1 \cdotcup G^2\) denotes the disjoint union of \(G^1\) and \(G^2\), that is the union of disjoint copies of both graphs, and \(K_n\) is the complete graph on \(n\) vertices. For all other notations we refer to \cite{diestel2005}.

Averbouch, Godlin and Makowsky \cite{averbouch2008, averbouch2010} state the following definition and results concerning the edge elimination polynomial.

\begin{defi}[Equation (13) in \cite{averbouch2008}]
\label{defi:eep}
Let \(G=(V, E), G^1, G^2\) be graphs and \(e \in E\) an edge of \(G\). The \emph{edge elimination polynomial} \(\xi(G, x, y, z)\) is defined as
\begin{align}
& \xi(G, x, y, z) = \xi(G_{-e}, x, y, z) + y \cdot \xi(G_{/e}, x, y, z) + z \cdot \xi(G_{\dagger e}, x, y, z), \label{eq:defi_eep_1} \\
& \xi(G^1 \cdotcup G^2, x, y, z) = \xi(G^1, x, y, z) \cdot \xi(G^2, x, y, z) \label{eq:defi_eep_2}, \\
& \xi(K_1, x, y, z) = x. \label{eq:defi_eep_3}
\end{align}
\end{defi}

The authors verify that the edge elimination polynomial is the most general invariant graph polynomial satisfying such recurrence relations \cite[Theorem 3 and Section 2]{averbouch2008}. Hence, it generalizes all graph polynomials obeying similar recurrence relations, including (the bivariate partition function of) the \emph{Potts model} \cite{sokal2005} and the \emph{chromatic polynomial} \cite{birkhoff1912, dong2005}. They give a list of such graph polynomials and the exact relations to the edge elimination polynomial \cite[Remark 4]{averbouch2008}. We omit this list here because we give analog results from the point of view of the covered components polynomial in Section \ref{sect:relations}.

Also the following expansion of the edge elimination polynomial is proven.

\begin{theo}[Theorem 5 in \cite{averbouch2008}]
\label{theo:eep_expansion}
Let \(G = (V, E)\) be a graph. Then the edge elimination polynomial \(\xi(G, x, y, z)\) satisfies
\begin{align*}
\xi(G, x, y, z) = \sum_{(A \sqcup B) \subseteq E}{x^{k(\esg{G}{A \cup B}) - c(\esg{G}{B})} y^{\abs{A} + \abs{B} - c(\esg{G}{B})} z^{c(\esg{G}{B})}},
\end{align*}
where \((A \sqcup B) \subseteq E\) is used for the summation over pairs of edge subsets \((A, B) \colon A, B \subseteq E\), such that the set of vertices incident to the edges of \(A\) and \(B\) are disjoint: \(\bigcup_{e \in A}{e} \cap \bigcup_{e \in B}{e} = \emptyset\).
\end{theo}

The fact that the edge elimination polynomial is the most general polynomial satisfying such a recurrence relation is investigated in more detail in the PhD thesis of Averbouch \cite[Section 3.3]{averbouch2010c}. Complexity results are given by Hoffmann \cite{hoffmann2010, hoffmann2010b}. Recently, White gives two other representation of the edge elimination polynomial for hypergraphs \cite{white2011} by defining the \emph{hyperedge elimination polynomial} and the \emph{trivariate chromatic polynomial}. This seems strongly related to the generalization for hypergraphs of the covered components polynomial given in Section \ref{sect:generalization_hypergraphs}.

%% file: definition.tex
\section{Definition and recurrence relations}

\label{sect:definition}

Motivated by the usage of the number of covered components in the expansion of the edge elimination polynomial given in Theorem \ref{theo:eep_expansion}, we define the covered components polynomial. This definition is in fact  an extension of the Potts model, which can be stated as generating function of the number of spanning subgraphs with respect to their number of components and edges.

\begin{defi}
\label{defi:ccp}
Let \(G = (V, E)\) be a graph. The \emph{covered components polynomial} \(C(G, x, y, z)\) is defined as
\begin{align}
& C(G, x, y, z) = \sum_{A \subseteq E}{x^{k(\esg{G}{A})} y^{\abs{A}}  z^{c(\esg{G}{A})}}.
\end{align}
\end{defi}

The covered components polynomial of a graph can also be defined as a sum over its set of spanning subgraphs. Let \(\esg{G}{\cdot}\) denote the set of spanning subgraphs of the graph \(G = (V, E)\), that is \(\esg{G}{\cdot} = \{\esg{G}{A} \mid A \subseteq E\}\). Then
\begin{align}
C(G, x, y, z) 
&= \sum_{H = (V, A) \in \esg{G}{\cdot}}{x^{k(H)} y^{\abs{A}} z^{c(H)}}. \label{eq:rema_def_ccp_2}
\end{align}
Furthermore, the covered components polynomial can be viewed as the ordinary generating function of the number of edge subsets \(A \subseteq E\) with \(k(\esg{G}{A}) = i\), \(\abs{A} = j\) and \(c(\esg{G}{A}) = k\), denoted by \(c_{i,j,k}(G)\):
\begin{align}
C(G, x, y, z) 
&= \sum_{i,j,k} c_{i,j,k}(G) x^{i} y^{j} z^{k}. \label{eq:rema_def_ccp_3}
\end{align}

\begin{theo}
\label{theo:ccp_rec}
Let \(G = (V,E), G^1, G^2\) be graphs and \(e \in E\) an edge of \(G\). Then the covered components polynomial \(C(G) = C(G, x, y, z)\) satisfies %the recurrence relations
\begin{align}
& C(G) = C(G_{-e}) + y \cdot C(G_{/e}) + (x y z - x y) \cdot C(G_{\dagger e}), \label{eq:c_rec1} \\
& C(G^1 \cdotcup G^2) = C(G^1) \cdot C(G^2) \label{eq:c_rec2}, \\
& C(K_1) = x. \label{eq:c_rec3}
\end{align}
\end{theo}

\begin{proof}
To prove Equation \eqref{eq:c_rec1}, the recurrence relation with respect to an edge \(e\), we consider the definition of the covered components polynomial using the enumeration over all spanning subgraphs stated in Equation \eqref{eq:rema_def_ccp_2}.

First, we consider all spanning subgraphs of \(G\) that do not contain edge \(e\). These subgraphs are exactly the spanning subgraphs of \(G_{-e}\) and hence
\begin{align*}
& C(G_{-e})
\end{align*}
enumerates all spanning subgraphs not including edge \(e\).

Second, we consider all spanning subgraphs of \(G\) that contain edge \(e\) but none of its adjacent edges. In this case, edge \(e\) and its incident vertices build a covered component. The rest of the spanning subgraph are exactly the spanning subgraphs of \(G_{\dagger e}\), hence these subgraphs are enumerated by \(C(G_{\dagger e}, x, y, z)\). The component built by edge \(e\) and its incident vertices contributes one component, one edge and one covered component to the polynomial, thus
\begin{align*}
& x y z \cdot C(G_{\dagger e})
\end{align*}
enumerates all spanning subgraphs including edge \(e\) but none of its adjacent edges.

Third, we consider all spanning subgraphs of \(G\) that contain edge \(e\) and at least one of its adjacent edges. These subgraphs are enumerated by \(y \cdot C(G_{/e})\), multiplied by \(y\) for the contracted edge \(e\). But this enumerated also spanning subgraphs that do not contain any edge incident to \(e\). To fix this, we subtract \(x y \cdot C(G_{\dagger e}, x, y, z)\) (multiplied by \(x\) because the edge with its incident vertices built a component in \(G_{/e}\), but not a covered component). Consequently,
\begin{align*}
& y \cdot C(G_{/e}) - x y \cdot C(G_{\dagger e})
\end{align*}
enumerates all spanning subgraphs including edge \(e\) and at least one of its adjacent edges.

Summing the three terms for the distinct cases we obtain Equation \eqref{eq:c_rec1}:
\begin{align*}
& C(G) = C(G_{-e}) + y \cdot C(G_{/e}) + (x y z - x y) \cdot C(G_{\dagger e}).
\end{align*}

To prove Equation \eqref{eq:c_rec2}, the multiplicativity with respect to components, we consider the graphs \(G = (V, E) = G^1 \cdotcup G^2\), \(G^1 = (V^1, E^1)\) and \(G^2 = (V^2, E^2)\). Each edge subset \(A \subseteq E\) is a disjoint union of edge subsets \(A^1 \subseteq E^1\) and \(A^2 \subseteq E^2\), where the number of components, edges and covered components in the subgraph of \(G\) spanned by \(A\) is the sum of this numbers in the subgraphs of \(G^1\) and \(G^2\) spanned by \(A^1\) and \(A^2\), respectively. More formally we have
\begin{align*}
C(G^1 \cdotcup G^2)
&= \sum_{A \subseteq E}{x^{k(\esg{G}{A})} y^{\abs{A}}  z^{c(\esg{G}{A})}} \\
&= \sum_{\substack{A^1 \subseteq E^1 \\ A^2 \subseteq E^2}}{x^{k(\esg{G^1}{A^1} \cdotcup \esg{G^2}{A^2})} y^{\abs{A^1 \cup A^2}} z^{c(\esg{G^1}{A^1} \cdotcup \esg{G^2}{A^2})}} \\
&= \sum_{\substack{A^1 \subseteq E^1 \\ A^2 \subseteq E^2}}{x^{k(\esg{G^1}{A^1}) + k(\esg{G^2}{A^2})} y^{\abs{A^1} + \abs{A^2}} z^{c(\esg{G^1}{A^1}) + c(\esg{G^2}{A^2})}} \\
&= \sum_{A^1 \subseteq E^1}{x^{k(\esg{G^1}{A^1})} y^{\abs{A^1}} z^{c(\esg{G^1}{A^1})}} \cdot \sum_{A^2 \subseteq E^2}{x^{k(\esg{G^2}{A^2})} y^{\abs{A^2}} z^{c(\esg{G^2}{A^2})}} \\
&= C(G^1) \cdot C(G^2).
\end{align*}

The statement of Equation \eqref{eq:c_rec3}, the initial values for the graph \(K_1\), follows directly from Definition \ref{defi:ccp}. The graph \(K_1\) has only one spanning subgraph, namely the subgraph spanned by the empty edge set \(A = \emptyset\), which consists of one component, no edges and no covered components. Hence we have
\begin{align*}
& C(K_1) = x. \qedhere
\end{align*}
\end{proof}

The following theorem states in fact the same result, but we give a totally different proof using the expansion of the edge elimination polynomial.

\begin{theo}
\label{theo:ccp_relation_eep}
Let \(G = (V, E)\) be graph, \(C(G, x, y, z)\) the covered components polynomial and \(\xi(G, x, y, z)\) the edge elimination polynomial. Then
\begin{align}
C(G, x, y, z) = \xi(G, x, y, x y z - x y).
\end{align}
\end{theo}

\begin{proof}
From the expansion of the edge elimination polynomial given in Theorem \ref{theo:eep_expansion} it follow that
\begin{align*}
\xi(G, x, y, x y z - x y) &= \sum_{(A \sqcup B) \subseteq E}{x^{k(\esg{G}{A \cup B}) - c(\esg{G}{B})} y^{\abs{A} + \abs{B} - c(\esg{G}{B})} (x y (z-1))^{c(\esg{G}{B})}} \\
&= \sum_{(A \sqcup B) \subseteq E}{x^{k(\esg{G}{A \cup B})} y^{\abs{A} + \abs{B}} (z-1)^{c(\esg{G}{B})}},
\end{align*} 
where \((A \sqcup B) \subseteq E\) is used for summation over pairs of edge subsets \((A, B) \colon A, B \subseteq E\), such that the set of vertices incident to the edges of \(A\) and \(B\) are disjoint.

We want to replace the summation over \((A \sqcup B) \subseteq E\) by a summation over \(F \subseteq E\), where \(F = A \cup B\). We observe, that almost all edge sets \(F \subseteq E\) are not only created by one but by a number of different pairs of edge sets \((A \sqcup B) \subseteq E\).

Following the definition of \((A \sqcup B) \subseteq E\), all edges of each covered component must be elements of either the edge set \(A\) or the edge set \(B\). Thus, an edge subset \(F \subseteq E\) with \(c(\esg{G}{F})\) covered components could be created by any edge set \(B\), such that \(B\) has exactly the edges of a subset of covered components of \(\esg{G}{F}\), and an edge set \(A\) with \(A = F \setminus B\). Consequently, the edge set \(B\) can be chosen so that \(c(\esg{G}{B}) = i\) for \(i \in \{0, \ldots, c(\esg{G}{F})\}\) (the set \(B\) has exactly the edges of \(i\) covered components of \(\esg{G}{F}\)) and for an edge set \(B\) with \(c(\esg{G}{B}) = i\) there are \(\binom{c(\esg{G}{F})}{i}\) distinct possibilities, to choose \(i\) among the \(c(\esg{G}{F})\) covered components. Hence we have
\begin{align*}
\xi(G, x, y, x y z - x y) &= \sum_{F \subseteq E}{\sum_{i=0}^{c(\esg{G}{F})}{\binom{c(\esg{G}{F})}{i} x^{k(\esg{G}{F})} y^{\abs{F}} \cdot (z-1)^{i}}} \\
&= \sum_{F \subseteq E}{x^{k(\esg{G}{F})} y^{\abs{F}} \sum_{i=0}^{c(\esg{G}{F})}{\binom{c(\esg{G}{F})}{i} (z-1)^{i}}} \\
&= \sum_{F \subseteq E}{x^{k(\esg{G}{F})} y^{\abs{F}} z^{c(\esg{G}{F})}} \\
&= C(G, x, y\, z). \qedhere
\end{align*} 
\end{proof}

As a corollary of the two theorems above we state that the covered components polynomial and the edge elimination polynomials are substitution instances of each other.

\begin{coro}
\label{coro:con_xi_ccp}
Let \(G = (V, E)\) be graph, \(C(G, x, y, z)\) the covered components polynomial and \(\xi(G, x, y, z)\) the edge elimination polynomial. Then
\begin{align}
& C(G, x, y, z) = \xi(G, x, y, x y z - x y), \\
& \xi(G, x, y, z) = C(G, x, y, \frac{z}{xy} + 1).
\end{align}
\end{coro}

For special structures in graphs, namely pendant edges, articulations and bridges, more specific recurrence relations can be given. An edge \(e \in E\) of a graph \(G = (V, E)\) is a pendant edge, if at least one of its incident vertices has degree \(1\).

\begin{coro}
\label{coro:ccp_pendant_edge}
Let \(G = (V, E)\) be a graph and \(e \in E\) a pendant edge of \(G\). Then the covered components polynomial \(C(G) = C(G, x, y, z)\) satisfies
\begin{align}
C(G) &= (x + y) \cdot C(G_{/e}) + (x y z - x y) \cdot C(G_{\dagger e}).
\end{align}
\end{coro}

\begin{proof}
If \(e\) is a pendant edge, then the graph \(G_{-e}\) is the disjoint union of the graph \(G_{/e}\) and a single vertex, consequently \(C(G_{-e}) = x \cdot C(G_{/e})\). Applying this to Equation \eqref{eq:c_rec1} we end up with the statement.
\end{proof}

For two graphs \(G^1 = (V^1, E^2)\) and \(G^2 = (V^2, E^2)\), we denote by \(G^1 \cup G^2\) the union and by \(G^1 \cap G^2\) the intersection of the graphs \(G^1\) and \(G^2\), that is \(G^1 \cup G^2 = (V^1 \cup V^2, E^1 \cup E^2)\) and \(G^1 \cap G^2 = (V^1 \cap V^2, E^1 \cap E^2)\).

A vertex \(v \in V\) of a graph \(G = (V, E)\) is an articulation, if \(k(G) < k(G_{-v})\), where \(G_{-v}\) is the graph \(G\) with the vertex \(v\) deleted, that is the vertex and all its incident edges are removed. Analogous, an edge \(e \in E\) of \(G\) is a bridge, if \(k(G) < k(G_{-e})\).

\begin{theo}
\label{theo:articulation_c}
Let \(G = (V, E)\), \(G^1 = (V^1, E^1)\) and \(G^2 = (V^2, E^2)\) be graphs, such that \(v \in V\) is an articulation, \(G^1 \cup G^2 = G\) and \(G^1 \cap G^2 = (\{v\}, \emptyset)\). Then the covered components polynomial \(C(G) = C(G, x, y, z)\) satisfies
\begin{align}
C(G) 
& = (\frac{1}{x z} + \frac{2}{x}) \cdot C(G^1) C(G^2) + (- \frac{1}{z} - 1) \cdot \left[C(G^1) C(G^2_{-v}) + C(G^1_{-v}) C(G^2) \right] \notag \\
& \eqspace + (\frac{x}{z} + x) \cdot C(G^1_{-v}) C(G^2_{-v}). \label{eq:theo_rec_ccp_art}
\end{align}
\end{theo}

\begin{proof}
For the proof we have to consider the distinct cases, where the vertex \(v\) lies in a covered component or is an isolated vertex in each of the subgraphs \(G^1\) and \(G^2\).

For this purpose we need the covered components polynomials for \(G^i\), \(i \in \{1,2\}\), for both situations. By \(G^{i'}\) we denote the subgraph of \(G^i\), where the edges incident to the vertex \(v\) are deleted. Thus \(C(G^{i'})\) enumerates the spanning subgraphs of \(G^{i}\), where the vertex \(v\) is an isolated vertex, and we have
\begin{align*}
C(G^{i'}) = x \cdot C(G^{i}_{-v}).
\label{eq:c_g'}
\end{align*}
Consequently, the spanning subgraphs of \(G^i\), where the vertex \(v\) lies in a covered component, are enumerated by
\begin{align*}
C(G^i) - C(G^{i'}) = C(G^i) - x \cdot C(G^i_{-v}).
\end{align*}

If the vertex \(v\) lies in a covered component in \(G^1\) and in \(G^2\), then the covered component (and so also the component) in \(G\) including the vertex \(v\) is counted twice, hence we have
\begin{align*}
\frac{[C(G^1) - x \cdot C(G^1_{-v})][C(G^2) - x \cdot C(G^2_{-v})]}{x z}
\end{align*}
in this case.

If the vertex \(v\) lies in a covered component in one subgraph and is an isolated vertex in the other subgraph, then only the component in \(G\) including the vertex \(v\) is counted twice, consequently we have
\begin{align*}
\frac{[C(G^1) -  x \cdot C(G^1_{-v})]C(G^2)}{x} + \frac{C(G^1)[C(G^2) - x \cdot C(G^2_{-v})]}{x}
\end{align*}
in this case.

If the vertex \(v\) is an isolated vertex in both subgraphs, then again only the component in \(G\) consisting only of the isolated vertex \(v\) is counted twice, so we have
\begin{align*}
\frac{x \cdot C(G^1_{-v}) \cdot x \cdot C(G^2_{-v})}{x} = x \cdot C(G^1_{-v}) C(G^2_{-v})
\end{align*}
in this case.

Summing up the single contributions we get the statement:
\begin{align*}
C(G) 
&= \frac{[C(G^1) - x \cdot C(G^1_{-v})][C(G^2) - x \cdot C(G^2_{-v})]}{x z} \\
& \eqspace + \frac{[C(G^1) -  x \cdot C(G^1_{-v})]C(G^2)}{x} + \frac{C(G^1)[C(G^2) - x \cdot C(G^2_{-v})]}{x} \\
& \eqspace + x \cdot C(G^1_{-v}) C(G^2_{-v}) \\
& = (\frac{1}{x z} + \frac{2}{x}) \cdot C(G^1) C(G^2) + (- \frac{1}{z} - 1) \cdot \left[C(G^1) C(G^2_{-v}) + C(G^1_{-v}) C(G^2) \right] \notag \\
& \eqspace + (\frac{x}{z} + x) \cdot C(G^1_{-v}) C(G^2_{-v}). \qedhere
\end{align*}
\end{proof}

Equation \eqref{eq:theo_rec_ccp_art} can be written in matrix form as
\begin{align}
C(G, x, y, z) = 
\begin{pmatrix}
C(G^1) & C(G^1_{-v})
\end{pmatrix}
\begin{pmatrix}
\frac{1}{x z} + \frac{2}{x} & - \frac{1}{z} - 1 \\
- \frac{1}{z} -1 & \frac{x}{z} + x
\end{pmatrix}
\begin{pmatrix}
C(G^2) \\
C(G^2_{-v})
\end{pmatrix}.
\end{align}

\begin{theo}
Let \(G = (V, E)\) be a graph and \(e = \{v_1, v_2\} \in E\) an bridge in \(G\). Then the covered components polynomial \(C(G) = C(G, x, y, z)\) satisfies
\begin{align}
C(G)
&= (1 + \frac{y}{x z} + \frac{2y}{x}) \cdot C(G_{-e}) + (- \frac{y}{z} - y) \cdot \left[C(G_{-v_1}) + C(G_{-v_2}) \right] \notag \\
& \eqspace + (\frac{x y}{z} + x y z) \cdot C(G_{\dagger e}). \label{eq:th_bridge_c}
\end{align}
\end{theo}

\begin{proof}
We begin by using the standard recurrence relation for the edge \(e = \{v_1, v_2\}\) and denote the vertex, to which the vertices \(v_1\) and \(v_2\) in \(G_{/e}\) are unified, by \(v\). Because edge \(e\) is a bridge, \(G_{/e}\) will result in a graph with an articulation at vertex \(v\). Hence, we can use the result of the Theorem \ref{theo:articulation_c}. 

Let \(G^1\) and \(G^2\) be the components of \(G_{-e}\), which includes the vertices \(v_1\) and \(v_2\), respectively. Then \(G^1\) and \(G^2\) match with the subgraphs \(G^1\) and \(G^2\) of Theorem \ref{theo:articulation_c}, where \(v_1\) and \(v_2\) represents the vertex \(v\) in \(G^1\) and \(G^2\), respectively. So we have:
\begin{alignat*}{3}
& C(G^1) C(G^2) &&= C(G_{-e}), \\
& C(G^1) C(G^2_{-v}) &&= C(G_{-v_2}), \\
& C(G^1_{-v}) C(G_{2}) &&= C(G_{-v_1}), \\
& C(G^1_{-v}) C(G^2_{-v}) &&= C(G_{\dagger e}).
\end{alignat*}

Using this, we get Equation \eqref{eq:th_bridge_c} from the general recurrence relation stated in Equation \eqref{eq:c_rec1}:
\begin{align*}
C(G) 
&= C(G_{-e}) + y \cdot C(G_{/e}) + (x y z - x y) \cdot C(G_{\dagger e}) \\
&= C(G_{-e}) + (\frac{y}{x z} + \frac{2y}{x}) \cdot C(G_{-e}) + (- \frac{y}{z} - y) \cdot \left[C(G_{-v_1}) + C(G_{-v_2}) \right] \\
& \eqspace + (\frac{x y}{z} + x y) \cdot C(G_{\dagger e}) + (x y z - x y) \cdot C(G_{\dagger e}) \\
&= (1 + \frac{y}{x z} + \frac{2y}{x}) \cdot C(G_{-e}) + (- \frac{y}{z} - y) \cdot \left[C(G_{-v_1}) + C(G_{-v_2}) \right] \\
& \eqspace + (\frac{x y}{z} + x y z) \cdot C(G_{\dagger e}). \qedhere
\end{align*}
\end{proof}

%% file: invariants.tex
\section{Invariants}

\label{sect:invariants}

As stated in Corollary \ref{coro:con_xi_ccp}, the covered components polynomial and the edge elimination polynomial can be calculated from each other. Hence, the covered components polynomial generalizes some well studied graph polynomials and contains all graph invariants coded in these. Consequently, there is a multitude of invariants - properties and numbers - of a graph coded in its covered components polynomial. We list some common graph invariants and some graph invariants, which seem to be specific for this polynomial.

For a polynomial \(P(x)\) with a representation \(P(x) = \sum_{i}{a_i x^i}\), we denote by \(\pdeg{x}{P(x)}\) the degree of \(x\) in \(P(x)\) and by \(\pcoef{x^i}{P(x)}\) the coefficient of \(x^i\) in \(P(x)\). We abbreviate \(\pcoef{x_1^i}{\pcoef{x_2^j}{P(x_1, x_2)}}\) to \(\pcoef{x_1^i x_2^j}{P(x_1, x_2)}\).

The three most common invariants of a graph \(G\), which can be determined from the most graph polynomials, are the number of vertices \(n(G)\), the number of edges \(m(G)\), and the number of connected components \(k(G)\). 
\begin{lemm}
Let \(G = (V, E)\) be a graph with covered components polynomial \(C(G) = C(G, x, y, z)\). Then
\begin{align}
n(G) &= deg_{x} C(G), \\
m(G) &= deg_{y} C(G), \\
k(G) &= deg_{x} \pcoef{y^{m(G)}}{C(G)}.
\end{align}
\end{lemm}

An graph invariant that seems specific to the covered components polynomial is the number of covered connected components \(c(G)\). Together with the number of components \(k(G)\), the number of isolated vertices \(i(G)\) can be determined.
\begin{lemm}
Let \(G = (V, E)\) be a graph with covered components polynomial \(C(G) = C(G, x, y, z)\). Then
\begin{align}
c(G) &= deg_{z} \pcoef{y^{m(G)}}{C(G)}, \\
i(G) &= k(G) - c(G) = deg_{x} \pcoef{y^{m(G)}}{C(G)} - deg_{z} \pcoef{y^{m(G)}}{C(G)}. \label{eq:invariants_i}
\end{align}
\end{lemm}

The girth \(g(G)\) of a graph \(G\) is the minimum cardinality of an edge subset which induces a cycle. By \(\# g(G)\) we denote the number of such edge subsets (inducing such a cycle of length \(g(G)\)).
\begin{lemm}
Let \(G = (V, E)\) be a graph with covered components polynomial \(C(G) = C(G, x, y, z)\). Then
\begin{align}
g(G) &= \min_{j}{\{[x^{n(G)-j+1} y^{j} z^1]C(G) > 0\}}, \\
\# g(G) &= [x^{n(G)-g(G)+1} y^{g(G)} z^1]C(G).
\end{align}
\end{lemm}
The girth concerns also cycles consisting of just one and two edges, namely loops and parallel edges. Consequently, a graph has no loops if and only if \(g(G) > 1\) and it has neither loops nor parallel edges, that means it is simple, if and only if \(g(G) > 2\).
\begin{lemm}
Let \(G = (V, E)\) be a graph with covered components polynomial \(C(G) = C(G, x, y, z)\). Then \(G\) is simple, if and only if
\begin{align}
g(G) &= \min_{j}{\{\pcoef{x^{n(G)-j+1} y^{j} z^1}{C(G)} > 0\}} > 2.
\end{align}
\end{lemm}
While the girth of simple graphs is coded in the chromatic polynomial \cite{dong2005}, the property of having parallel edges and so also of being simple is not. With respect to the girth and the simplicity the Potts model contains the same information as the covered components polynomial.

Considering edge subsets \(A\) of minimum cardinality, such that the subgraph spanned by the edge subset \(E \setminus A\) has at least one isolated vertex, we can determine the minimum degree \(\delta(G)\) and the number of vertices with minimum degree, denoted by \(\# \deg_{\delta}(G)\).
\begin{lemm}
Let \(G = (V, E)\) be a simple graph with covered components polynomial \(C(G) = C(G, x, y, z)\). Then
\begin{align}
\delta(G) &= \min_{j}{\{\pcoef{x^i y^{m-j} z^k}{C(G)} > 0 \mid i > k\}}, \\
\# \deg_{\delta}(G) &= \sum_{i > k}{\pcoef{x^i y^{m-\delta(G)} z^k}{C(G)}} \qquad \text{ for } \delta(G) \neq 1.
\end{align}
\end{lemm}

It is also possible to get the number of vertices with degree \(1\), denoted by \(\# \deg_{1}(G)\).

\begin{theo}
\label{theo:invariants_d1}
Let \(G = (V, E)\) be a simple graph with covered components polynomial \(C(G) = C(G, x, y, z)\). Then
\begin{align}
\# \deg_{1}(G) &= \pcoef{x^{k(G)+1} y^{m(G)-1} z^{c(G)}}{C(G)} \notag \\
& \eqspace + 2 \cdot \pcoef{x^{k(G)+1} y^{m(G)-1} z^{c(G)-1}}{C(G)}.
\end{align}
\end{theo}

\begin{proof}
To determine the number of vertices of degree \(1\) we count the number of edges incident to one vertex and to two vertices with degree \(1\). To do so, we count the subgraphs spanned by edge subsets, where exactly these edges are missing. If and only if exactly one edge incident to one vertex of degree \(1\) is missing, the induced subgraph consists of \(k(G)+1\) components and \(c(G)\) covered components. Hence we can count these edges by
\begin{align*}
& \pcoef{x^{k(G)+1} y^{m(G)-1} z^{c(G)}}{C(G)}.
\end{align*}
If and only if exactly one edge incident to two vertices of degree \(1\) is missing (this edge together with its incident vertices built a covered component, which is a \(K_2\)), then the induced subgraph consists of \(k(G)+1\) components and \(c(G)-1\) covered components. Hence we can count these edges by
\begin{align*}
& \pcoef{x^{k(G)+1} y^{m(G)-1} z^{c(G-1)}}{C(G)}. \qedhere
\end{align*}
\end{proof}

For a graph \(G = (V, E)\) and an edge subset \(A \subseteq E\) of \(G\), \(\eig{G}{A} = (\bigcup_{e \in A}{\{e\}}, A)\) is the by \(A\) edge-induced subgraph of \(G\). The covered components polynomial counts the number of some edge-induced subgraphs.

\begin{theo}
Let \(G = (V, E)\) be a graph with covered components polynomial \(C(G) = C(G, x, y, z)\) and let \(g(G, n',m',k')\) be the number of edge-induced subgraphs with \(n'\) vertices, \(m'\) edges and \(k'\) components. Then
\begin{align}
& g(G, n', m', k') = \pcoef{x^{n(G)-n'+k'} y^{m'} z^{k'}}{C(G)}.
\end{align}
\end{theo}

\begin{proof}
As the edge-induced subgraph \(G'= \eig{G}{A}\) has \(n'\) vertices, \(m'\) edges and \(k'\) components, the spanning subgraph \(\esg{G}{A}\) must have \(n(G) - n'\) isolated vertices building components every for itself, \(n(G)-n'+k'\) components, \(k'\) of them covered components, and \(m'\) edges. This is the coefficient of \(x^{n(G)-n'+k'} y^{m'} z^{k'}\) in \(C(G)\).
\end{proof}

By these theorem it is possible to count the edge subsets inducing a graph, which has a unique parameter set of \(n'\) vertices, \(m'\) edges and \(k'\) components. For example, by
\begin{align}
& g(G, n', n'-1, 1) = \pcoef{x^{n(G)-n'+1'} y^{n'-1} z^{1}}{C(G)} \label{eq:20110307a}
\end{align}
we count trees (connected graphs with \(n'\) vertices and \(n'-1\) edges) and by 
\begin{align}
& g(G, n', \binom{n'}{2}, 1) = \pcoef{x^{n(G)-n'+1} y^{\binom{n'}{2}} z^{1}}{C(G)}
\end{align}
we count complete subgraphs (connected graphs with \(n'\) vertices and \(\binom{n'}{2}\) edges). The clique number \(\omega(G)\) of a simple graph \(G\) is the maximum number of vertices in a complete subgraph.

\begin{coro}
Let \(G = (V, E)\) be a simple graph with covered components polynomial \(C(G) = C(G, x, y, z)\). 
Then
\begin{align}
& \omega(G) = \max_{j}{\{g(G, j,\binom{j}{2}, 1) > 0\}}.
\end{align}
\end{coro}

As a special case of Equation \eqref{eq:20110307a} we can count the number of edge-induced paths of length \(2\) (connected graphs with \(3\) vertices and \(2\) edges), denoted by \(\#P_3(G)\). This is related to the \(M_1\)-index in combinatorial chemistry \cite{li2003}, defined as
\begin{align}
M_1(G) = \sum_{v \in V}{(\deg v)^2}.
\end{align}

\begin{coro}
\label{coro:m_1}
Let \(G = (V, E)\) be a simple graph with covered components polynomial \(C(G) = C(G, x, y, z)\). Then 
\begin{align}
& \#P_3(G) = g(G, 3, 2, 1) = \sum_{v \in V}{\binom{\deg v}{2}}, \\
& M_1(G) = 2 \cdot \#P_3(G) + 2 \cdot m(G).
\end{align}
\end{coro}

%% file: unique_graphs.tex
\section{Some \texorpdfstring{$C$}{C}-unique graphs}

\label{sect:unique_graphs}

\begin{defi}
Let \(G = (V, E)\) be a graph with covered components polynomial \(C(G, x, y, z)\). A graph \(G\) is \(C\)\emph{-unique}, if it is the only graph (up to isomorphism) with covered components polynomial \(C(G, x, y, z)\). 
\end{defi}

For many graph classes it is easy to deduce the \(C\)-uniqueness from their uniqueness with respect to a graph polynomial generalized by the covered components polynomial, which are given in Section \ref{sect:relations}. Cycles and complete graphs are \(\chi\)-unique \cite{koh1990}, that is unique with respect to the chromatic polynomial, and so also \(C\)-unique. Wheels and complete multipartite graphs are \(T\)-unique \cite{mier2004}, that is unique with respect to the Tutte polynomial and Potts model, and so also \(C\)-unique.

We restate the \(C\)-uniqueness for cycles, paths, complete graphs and extend this property to stars by an easy proof using the invariants mentioned in Section \ref{sect:invariants}.

\begin{theo}
The cycles \(C_n\), the paths \(P_n\), the stars \(S_n\) (with \(n+1\) vertices) and the complete graphs \(K_n\) are \(C\)-unique.
\end{theo}
\begin{proof}
We can conclude that a graph \(G\) is \(C\)-unique if we find some invariants coded in its covered components polynomial \(C(G) = C(G, x, y, z)\), which identify the graph uniquely.

We use the following invariants presented in Section \ref{sect:invariants}: number of vertices, number of edges, number of (connected) components, number of vertices of degree 1, and simplicity.

If we have a covered components polynomial \(C(G)\), from which we can extract, that \(G\) has \(n\) vertices, \(n\) edges, one component and minimum degree \(2\), then \(G\) is isomorph to \(C_n\).

If we have a covered components polynomial \(C(G)\), from which we can extract, that \(G\) has \(n\) vertices,  \(n-1\) edges and one component, then \(G\) is a tree with \(n\) vertices. If we further determine, that \(G\) has exactly \(2\) vertices of degree \(1\), then we know that \(G\) is isomorph to \(P_n\).

If we have a covered components polynomial \(C(G)\), from which we can extract, that \(G\) has \(n+1\) vertices, \(n\) edges, one component and \(n\) vertices with degree \(1\), then we know that \(G = S_n\).

If we have a covered components polynomial \(C(G)\), from which we can extract, that \(G\) has \(n\) vertices, \(\binom{n}{2}\) edges and is simple, than we know that \(G\) is isomorph to \(K_n\).
\end{proof}

A class of graphs, which is in general not \(\chi\)-unique, are the wheels \(W_n\) with \(n+1\) vertices. For even \(n\), the wheels \(W_n\) are \(\chi\)-unique, while there are some open questions for odd \(n\): \(W_5\) and \(W_7\) are not \(\chi\)-unique, \(W_9\) is \(\chi\) unique and for odd \(n > 9\) nothing is known \cite{koh1990, noy2003}. In \cite[Theorem 3.1]{mier2004} the authors proved the \(T\)-uniqueness for all wheels. In the case of \(C\)-uniqueness we can shorten this proof by usage of some additional ``degree invariants'' and some properties of the chromatic polynomial.
\begin{theo}
\label{theo:uniqueness_wheels}
The wheels \(W_n\) are \(C\)-unique.
\end{theo}
\begin{proof}
Assume that we have a graph \(H\) with the same covered components polynomial as the wheel \(W_n\), that is \(C(H) = C(W_n) = C\). From \(C\) we can determine the following properties (mentioned in Section \ref{sect:invariants}) of the corresponding graphs: the number of vertices, the number of edges, the minimum degree, the number of vertices with minimum degree and the simplicity. Because \(W_n\) has \(n+1\) vertices, \(2 n\) edges, minimum degree \(3\), \(n\) vertices with minimum degree and is simple, the same is true for \(H\) and hence it has one vertex adjacent to all other vertices.
Further more, if \(H\) and \(W_n\) have the same covered components polynomial, they also have the same chromatic polynomial. For a graph \(G=(V, E)\) with a vertex \(v \in V\) adjacent to all other vertices, the chromatic polynomial satisfies \cite[Corollary 1.5.1]{dong2005}:
\begin{align*}
\chi(G, x) = x \cdot \chi(G_{-v}, x-1). 
\end{align*}
Hence, if we denote the vertex adjacent to all other vertices in \(W_n\) and \(G\) by \(v\), we have:
\begin{align*}
& \chi(W_n, x) = x \cdot \chi(W_{n,-v}, x-1), \\
& \chi(H, x) = x \cdot \chi(H_{-v}, x-1), \\
\intertext{from which we can deduce that}
& \chi(W_{n,-v}, x-1) = \chi(H_{-v}, x-1).
\end{align*}
\(W_{n,-v}\) is the cycle with \(n\) vertices \(C_n\), which is chromatically unique \cite{koh1990, noy2003}. Thus \(H_{-v}\) is isomorphic to \(C_n\). Consequently \(H\) (and any other graph with same covered components polynomial as \(W_n\)) is isomorphic to \(W_n\) and so \(W_n\) is \(C\)-unique.
\end{proof}

%% file: special_graphs.tex
\section{Special graphs}

\label{sect:special_graphs}

We give recurrence relations and explicit formulas paths, cycles, stars and complete graphs.

\begin{lemm}
\label{lemm:cnik}
Let \(P_n = (V,E)\) be the path with \(n\) vertices and \(c(n, i, k)\) be the number of subgraphs of \(P_n\) spanned by an edge subset \(A \subseteq E\) with \(i\) edges and \(k\) covered components. Then for \(n, i, k > 0\) it holds
\begin{align}
c(n, i, k) = \binom{i-1}{k-1} \binom{n-i}{k}. \label{eq:cnik1}
\end{align}
Otherwise (\(n = 0\) or \(i = 0\) or \(k = 0\)) it is
\begin{align}
c(n, i, k) =
\begin{cases}
1 & \text{if } i = k = 0, \\
0 & \text{otherwise.} \\
\end{cases}
\label{eq:cnik2}
\end{align}
\end{lemm}

\begin{proof}
A path of \(n\) vertices has \(n-1\) edges, from which we want to distribute \(i\) edges over \(k\) covered components, and to distribute the remaining \(n-1-i\) (non-existing) edges so, that the \(k\) covered components are separated.

We first consider the composition of \(i\) edges in \(k\) covered components. The covered components can be numbered consecutively by their occurrence in a path between the two end vertices. Hence, the covered components are distinguishable and the composition of the edges is equal to an ordered integer partition of \(i\) with \(k\) parts, which can be selected in \(\binom{i-1}{k-1}\) different ways.

From the remaining (non-existing) edges we use \(k-1\) to separate the \(k\) covered components, that is these are inserted between each two covered components. So we have a remaining of \(n-i-k\) edges to distribute at \(k+1\) places, where \(k-1\) places are between each two covered components and one place is at each end of the path. This equals the distribution of \(n-i-k\) non distinguishable objects into \(k+1\) distinguishable boxes, hence we have \(\binom{n-i}{k}\) possibilities.

Because these two selections are performed independently, we multiply the two terms and get Equation \eqref{eq:cnik1}.

If we have \(i = 0\) edges it is only possible to have \(k = 0\) covered components, for which we have exactly one possibility, so we get Equation \eqref{eq:cnik2}.
\end{proof}

\begin{theo}
Let \(P_n\) be the path with \(n\) vertices. Then the covered components polynomial \(C(P_n) = C(P_n, x, y, z)\) for \(n > 1\) satisfies the recurrence relation
\begin{align}
& C(P_n) = (x + y) \cdot C(P_{n-1}) + (x y z - x y) \cdot C(P_{n-2}) \label{eq:p1}
\end{align}
with the initial conditions \(C(P_0) = 1\), \(C(P_1) = x\) and the explicit formula
\begin{align}
& C(P_n) = x^{n} + \sum_{i=1}^{n-1}{x^{n-i} y^{i} \sum_{k=1}^{\min(i, n-i)}{\binom{i-1}{k-1} \binom{n-i}{k} z^{k}}}. \label{eq:p4}
\end{align}
\label{theo:c_p}
\end{theo}

In Equation \eqref{eq:p4} the upper limit of the inner sum also could be omitted, because for \(k > \min(i, n-1)\) one of the binomial coefficients is \(0\). To prove Theorem \ref{theo:c_p} we give a lemma which makes the proof really easy.

\begin{proof}
The recurrence relation stated in Equation \eqref{eq:p1} can be proved by using the recurrence relations given by Equation \eqref{eq:c_rec1} in Theorem \ref{theo:ccp_rec} for an edge \(e\), incident to a vertex with degree \(1\). For \(P_2\) there exists exactly one and for every \(P_n\) with \(n \geq 3\) there exist exactly two such edges \(e\). The initial conditions can be verified from Definition \ref{defi:ccp}.

To prove the explicit formula stated in Equation \eqref{eq:p4} we just have to show that this formula satisfies the recurrence relations given by Equation \eqref{eq:p1} and the initial conditions. We apply Lemma \ref{lemm:cnik}. 

We first consider an edge subset \(A = \emptyset \subseteq E(P_n)\) with \(0\) edges. Then the induced subgraph has \(n\) components, no edges and no covered components which yields the term \(x^{n}\).

Next we consider an edge subset \(A \subseteq E(P_n)\) with exactly \(i \in \{1,\ldots,n-1\}\) edges. Because each edge reduces the number of components by \(1\), the induced subgraph has \(n-i\) components.

The edge set \(A\) has at least one edge, so the induced subgraph has at least one covered component. The number of covered components is bounded from above by the number of edges \(i\) (each covered component has at least one edge) and by \(n-i\). Because having \(n-1\) edges, we need at least \(k-1\) edges to separate \(k\) covered components in a path \(P_n\) and so we have at most \(n-1\) edges for \(n-1\) covered components.

According to Lemma \ref{lemm:cnik} the number of subgraphs of the \(P_n\) with \(i\) edges and \(k\) covered components is
\begin{align*}
c(n, i, k) = \binom{i-1}{k-1} \binom{n-i}{k},
\end{align*}
because we only consider the case \(n > 0\), \(i > 0\) and \(k > 0\).

Putting the parts together, we get Equation \eqref{eq:p4}.
\end{proof}

\begin{theo}
\label{theo:c_c}
Let \(C_n\) be the cycle with \(n\) vertices. Then the covered components polynomial \(C(C_n) = C(C_n, x, y, z)\) for \(n > 3\) satisfies the recurrence relation
\begin{align}
& C(C_n) = C(P_n) + y \cdot C(C_{n-1}) + (x y z - x y) \cdot C(P_{n-2}), \label{eq:c1}
\end{align}
and the initial condition
\begin{align}
& C(C_3) = x^{3} + 3 x^{2} y z + 3 x y^{2} z + x y^{3} z,  \label{eq:c2}
\end{align}
and for \(n > 2\) the explicit formula 
\begin{align}
C(C_n) = x^{n} + \sum_{i=1}^{n-1}{\frac{n}{n-i} x^{n-i} y^{i} \sum_{k=1}^{\min(i, n-i)}{\binom{i-1}{k-1} \binom{n-i}{k} z^{k}}} + x y^{n} z. \label{eq:c3}
\end{align}
\end{theo}

\begin{proof}[Proof of Theorem \ref{theo:c_c}.]
The recurrence relation stated in Equation \eqref{eq:c1} can be proved by using the recurrence relations given by Equation \eqref{eq:c_rec1} in Theorem \ref{theo:ccp_rec} for any edge \(e\). The initial condition \eqref{eq:c2} can be verified from Definition \ref{defi:ccp}.

To prove the explicit formula stated in Equation \eqref{eq:c3} we just have to show that this formula satisfies the recurrence relations given by Equations \eqref{eq:c1} and \eqref{eq:c2}. We apply Lemma \ref{lemm:cnik}. 

First we consider the minimal and the maximal edge subsets \(A \subseteq E(C_n)\). If \(A = \emptyset\), then the induced subgraph has \(n\) components, \(0\) edges and \(0\) covered components, which brings the term \(x^{n}\). If \(A = E\), then the induced subgraph has one component, \(n\) edges and covered component, which brings the term \(x y^{n} z\).

In all other cases we have an edge subset \(A \subseteq E(C_n)\) with exactly \(i \in \{1,\ldots,n-1\}\) edges. Because we have at least one non-existing edges (\(n-i \ge 1\)), we can choose a edge as a non-existing edge and delete it from the graph. The remaining graph is a path with \(n\) vertices, where we can determine the number of edge subsets with \(i\) edges and \(k\) covered components according to Lemma \ref{lemm:cnik} by \(c(n, i, k)\).

To choose a non-existing edge we have \(n\) possibilities. But every edge subset with \(i\) edges occurs once, if we choose one of the \(n-i\) not existing edges. In summary, we have to multiply \(c(n, i, k)\) by \(\frac{n}{n-i}\).

Putting the parts together, we get Equation \eqref{eq:c3}.
\end{proof}

\begin{rema}
If we define the cycle \(C_n\) as a path \(P_n\), where the first vertex and the last vertex are connected by an additional edge (that means the \(C_1\) is a vertex with a loop and the \(C_2\) are \(2\) vertices connected by \(2\) parallel edges), then the recurrence relation and the explicit formula also holds for \(n > 1\) with \(C(C_1) = x + x y z\) and \(n > 0\), respectively.
\end{rema}

From Equation \eqref{eq:c3} we can determine the coefficient of \(x^{n-i} y^{i} z^{k}\) in \(C(C_{n})\) for \(n, i, k > 0\), \(n \neq i\):
\begin{align}
\pcoef{x^{n-i} y^{i} z^{k}}{C(C_{n})} = \frac{n}{n-i} \binom{i-1}{k-1} \binom{n-i}{k}.
\end{align}
For \(i = k\) we get
\begin{align}
\pcoef{x^{n-k} y^{k} z^{k}}{C(C_{n})} = \frac{n}{n-k} \binom{n-k}{k},
\end{align}
which is the number of \(k\)-matchings in a \(C_n\) and the number of choosing \(k\) points, no two of them consecutive, from a collection of \(n\) points arranged in a circle \cite[Lemma 2.3.4]{stanley1986}. This is because with \(i=k\) each covered component consists of exactly one edge, which corresponds to the chosen point.

\begin{theo}
Let \(S_n\) be the star with \(n+1\) vertices. Then the covered components polynomial \(C(S_n) = C(S_n, x, y, z)\) for \(n > 0\) satisfies the recurrence relation
\begin{align}
C(S_n) = (x + y) \cdot C(S_{n-1}) + (y z - z) x^{n}, \label{eq:s1}
\end{align}
and the initial condition
\begin{align}
C(S_0) = C(K_1) = x, \label{eq:s2}
\end{align}
and the explicit formula
\begin{align}
C(S_n) &= x^{n+1} + \sum_{j=1}^{n}{\binom{n}{j} x^{n+1-j} y^{j} z}, \label{eq:s3} \\
&= x^{n+1} (1-z) + x z (x + y)^{n}. \label{eq:s4}
\end{align}
\end{theo}

\begin{proof}
The recurrence relation \eqref{eq:s1} can be proofed by using the recurrence relation stated in Equation \eqref{eq:c_rec1} for any edge \(e\). The initial condition \eqref{eq:s2} can be verified from the Definition \ref{defi:ccp}.

To proof the explicit formula \eqref{eq:s3} we consider the spanning subgraphs induced by \(j\) edges. If \(j=0\) we have the \(E_{n+1}\), the edgeless graph with \(n+1\) vertices, which gives the term \(x^{n+1}\). For \(j > 0\) we have \(\binom{n}{j}\) possibilities to select the edges, where all subgraphs induced by \(j\) edges are isomorphic to each other. They consist of a star \(S_j\) (with the chosen \(j\) edges) and \(n-j\) isolated vertices, which gives the term \(x^{n+1-j} y^{j} z^{1}\) for each subgraph.

Rearranging formula \eqref{eq:s3} we get formula \eqref{eq:s4}.
\end{proof}

\begin{theo}
Let \(K_n = (V, E)\) be the complete graph with \(n\) vertices and \(d(i, j)\) the number of connected spanning subgraphs of \(K_i\) with exactly \(j\) edges. Then the covered components polynomial \(C(K_n) = C(K_n, x, y, z)\) for \(n > 0\) satisfies the recurrence relation
\begin{align}
C(K_n) = x \cdot C(K_{n-1}) + \sum_{i=2}^{n}{\left[ \binom{n-1}{i-1} \ \sum_{j=i-1}^{\binom{i}{2}}{d(i, j) \cdot x y^{j} z \cdot C(K_{n-i})} \right]}.
\label{eq:c_k_n}
\end{align}
\end{theo}

\begin{proof}
To prove this result we choose a vertex \(v \in V\) and determine the covered components polynomial for the connected component also including the vertex \(v\) and for the rest of the graph. By the multiplicativity with respect to components we then get the covered components polynomial of the graph \(K_n\).

As mentioned, we choose a vertex \(v \in V\). If \(v\) is an isolated vertex, then the covered components polynomial of the rest of the graph can be calculated by \(C(K_{n-1})\), thus we have
\begin{align*}
x \cdot C(K_{n-1}).
\end{align*}
If \(v\) is not an isolated vertex, then it is together with \(i-1\) for \(i \in \{2, \ldots, n\}\) other vertices in a covered component. To choose this \(i-1\) vertices out of the remaining \(n-1\) vertices we have \(\binom{n-1}{i-1}\) options and to connect these \(i\) vertices to a component we have \(d(i, j)\) options by using \(j\) edges for \(j \in \{i-1, \ldots, \binom{i}{2}\}\). The covered components polynomial of the rest of the graph can be calculated by \(C(K_{n-i})\), thus we have
\begin{align*}
\sum_{i=2}^{n}{\left[ \binom{n-1}{i-1} \ \sum_{j=i-1}^{\binom{i}{2}}{d(i, j) \cdot x y^{j} z \cdot C(K_{n-i})} \right]}.
\end{align*}
Summing the two terms for the distinct cases we obtain Equation \eqref{eq:c_k_n}.
\end{proof}

The number of connected spanning subgraphs of the \(K_i\) with \(j\) edges, \(d(i, j)\) equals the number of connected graphs on \(i\) vertices with \(j\) edges. For this, Simon \cite[Equation (3.90)]{simon2009} shows that
\begin{align}
d(i, j) = \binom{\binom{i}{2}}{j} - \sum_{k=0}^{i-1}{\sum_{l=0}^{j}{d(k, l) \binom{i-1}{k-1} \binom{\binom{i-k}{2}}{j-l}}}.
\end{align}

%% file: relations.tex
\section{Relations to other graph polynomials and \texorpdfstring{\\}{} distinctive power}

\label{sect:relations}

We give an overview of the calculations, which must be carried out to get some well known other graph polynomials from the covered components polynomial. For the edge elimination polynomial this is given in \cite[Remark 4]{averbouch2008}. While the covered components polynomial is defined also for graphs with parallel edges and loops, some of the graph polynomials we mention in the following are not. Hence, the given relations contain implicit generalizations for these graph polynomials. As stated in Corollary \ref{coro:con_xi_ccp}, for the edge elimination polynomial we have:
\begin{itemize}
	\item edge elimination polynomial \cite{averbouch2008}:
\begin{align}
		& \xi(G, x, y, z) = C(G, x, y, \frac{z}{x y} + 1), \\
		& C(G, x, y, y) = \xi(G, x, y, x y z - x y).
\end{align}
\end{itemize}
For the definition of (the bivariate partition function of) the Potts model \cite[Equation 1.1]{sokal2005} it follows:
\begin{itemize}
	\item Potts model \cite{sokal2005}:
	\begin{align}
		& Z(G, x, y) = C(G, x, y, 1).
	\end{align}
\end{itemize}
From the recurrence relations for the chromatic polynomial \cite[Equation 1.12]{dong2005} and the bivariate chromatic polynomial \cite[Proposition 1]{averbouch2008} it follows:
\begin{itemize}
	\item chromatic polynomial \cite{dong2005}:
\begin{align}	
		& \chi(G, x) = C(G, x, -1, 1), \\
	\intertext{\item bivariate chromatic polynomial \cite{dohmen2003}:}
		& P(G, x, y) = C(G, x, -1, \frac{y}{x}). \label{eq:relations_bcp_as_ccp}
\end{align}
\end{itemize}
In \cite[Corollary 2]{dohmen2003} a relation between the bivariate chromatic polynomial and another polynomial, called ``independence polynomial'' is given. In fact, this polynomial is not the independence polynomial defined as usual \cite{levit2005}, because it ``counts'' the number of vertices not in the independent set instead of the number of vertices in the independent set. But for every edge at least one incident vertex is not in the independence set, hence this other polynomial is exactly the vertex-cover polynomial \cite{dong2002}. Using the connection between the vertex-cover polynomial and the independence polynomial we give relations for both polynomials:
\begin{itemize}
	\item vertex-cover polynomial \cite{dong2002}:
\begin{align}
\Psi(G, x) &= C(G, x+1, -1, \frac{1}{x+1}) \label{eq:psi_as_c},
\end{align}
\item independence polynomial \cite{levit2005}:
\begin{align}
I(G, x) 
&= x^{n(G)} \cdot \Psi(G, \frac{1}{x}) \\
&= x^{n(G)} \cdot C(G, \frac{x+1}{x}, -1, \frac{x}{x+1}) \label{eq:i_as_c}.
\end{align}
\end{itemize}
Using the definition of the bivariate matching polynomial and the clique polynomial and also the recurrence relation for the first, we get:
\begin{itemize}
	\item bivariate matching polynomial \cite{averbouch2008}:
\begin{align}
M(G, x, y) 
&= \sum_{i = 0}^{m(G)}{[x^{n(G)-i} y^{i} z^{i}]C(G, x, y, z) \cdot x^{n(G)-2i} y^{i}} \\
&= \left. C(G, x, z, \frac{y}{x z}) \right|_{z=0} \\
&= \xi(G, x, 0, y),
\end{align}
\item clique polynomial \cite{hoede1994} (only for simple graphs):
\begin{align}
C(G, x) &= 1 + n(G) \cdot x + \sum_{i = 2}^{n(G)}{ \left[ x^{n(G)-i+1} y^{\binom{i}{2}} z^{1} \right] C(G, x, y, z) \cdot x^{i}}.
\end{align}
\end{itemize}

There is also a relation which connects the covered components polynomial of a graph \(G =(V,E)\) with the subgraph component polynomial of its line graph \(L(G)\), which is stated in \cite[Theorem 23]{tittmann2011}:
\begin{itemize}
	\item subgraph component polynomial \cite{tittmann2011}:
\begin{align}
Q(L(G), x, y) &= C(G, 1, x, y).
\end{align}
\end{itemize}

Using the relations and the definition of the covered components polynomial \(C(G, x, y, z)\), we can give representations as sums over edge subset for the edge elimination polynomial, the bivariate chromatic polynomial, the vertex-cover polynomial and the independence polynomial. These are given in terms of the number of components \(k(\esg{G}{A})\), number of covered components \(c(\esg{G}{A})\), and number of isolated vertices \(i(\esg{G}{A})\) of the subgraphs \(\esg{G}{A}\) spanned by the edge subsets \(A\).

\begin{coro}
Let \(G=(V, E)\) be a graph with edge elimination polynomial \(\xi(G, x, y, z)\). Then
\begin{align}
\xi(G, x, y, z) 
&= C \left(G, x, y, \frac{z}{xy}+1 \right) \\
&= \sum_{A \subseteq E}{x^{k(\esg{G}{A})} \ y^{\abs{A}} \ \left( \frac{z}{xy} + 1 \right)^{c(\esg{G}{A})}} \notag \\
&= \sum_{A \subseteq E}{x^{i(\esg{G}{A})} \ y^{\abs{A} - c(\esg{G}{A})} \ (z + xy)^{c(\esg{G}{A})}}.
\end{align}
\end{coro}

\begin{coro}
Let \(G = (V, E)\) be a graph with bivariate chromatic polynomial \(P(G, x, y)\). Then
\begin{align}
P(G, x, y) = C \left(G, x, -1, \frac{y}{x} \right) &= \sum_{A \subseteq E}{x^{k(\esg{G}{A})} \ (-1)^{\abs{A}} \ \left( \frac{y}{x} \right)^{c(\esg{G}{A})}} \notag \\
&= \sum_{A \subseteq E}{(-1)^{\abs{A}} \ x^{i(\esg{G}{A})} \ y^{c(\esg{G}{A})}}.
\end{align}
\end{coro}
The equality above is strongly related to the representation of the bivariate chromatic polynomial as a sum over connected partitions \cite[Theorem 4]{dohmen2003}.

\begin{coro}
Let \(G = (V, E)\) be a graph with vertex-cover polynomial \(\Psi(G, x)\). Then
\begin{align}
\Psi(G, x) 
&= C \left(G, x+1, -1, \frac{1}{x+1} \right) \\
&= \sum_{A \subseteq E}{(x+1)^{k(\esg{G}{A})} \ (-1)^{\abs{A}} \ \left( \frac{1}{x+1} \right)^{c(\esg{G}{A})}} \notag \\
&= \sum_{A \subseteq E}{(-1)^{\abs{A}} \ (x+1)^{i(\esg{G}{A})}}.
\end{align}
Let \(G^1, G^2\) be graphs and \(e \in E\) an edge of \(G\). Then the vertex-cover polynomial \(\Psi(G, x)\) satisfies
\begin{align}
& \Psi(G, x) = \Psi(G_{-e}, x) - \Psi(G_{/e}, x) + x \cdot \Psi(G_{\dagger e}, x), \\
& \Psi(G^1 \cdotcup G^2, x)= \Psi(G^1, x) \cdot \Psi(G^2, x), \\
& \Psi(K_1, x) = x+1.
\end{align}
\end{coro}

\begin{coro}
Let \(G=(V, E)\) be a graph with independence polynomial \(I(G, x)\). Then
\begin{align}
I(G, x)
&= x^{n(G)} \cdot C \left(G, \frac{x+1}{x}, -1, \frac{x}{x+1} \right) \notag \\
&= x^{n(G)} \sum_{A \subseteq E}{(-1)^{\abs{A}} \left( \frac{x+1}{x} \right)^{i(\esg{G}{A})}} \notag \\
&= \sum_{A \subseteq E}{(-1)^{\abs{A}} x^{n(G)-i(\esg{G}{A})} (x+1)^{i(\esg{G}{A})}}
\end{align}
and for the number of independent sets \(\sigma(G) = I(G, 1)\) it holds
\begin{align}
\sigma(G) = \sum_{A \subseteq E}{(-1)^{\abs{A}} 2^{i(\esg{G}{A})}}.
\label{eq:rema_sigma}
\end{align}
\end{coro}
A direct derivation for Equation \eqref{eq:rema_sigma} using inclusion and exclusion is given in \cite[Equations (II.14) - (II.16)]{merrifield1981}.

A graph polynomial \(P(G)\) distinguishes two graphs \(G^1\) and \(G^2\), if \(P(G^1) \neq P(G^2)\). We say a graph polynomial \(P(G)\) has at least the distinctive power of the graph polynomial \(Q(G)\), if each pair of graphs distinguished by \(P(G)\) is also distinguished by \(Q(G)\), that is
\begin{align}
P(G^1) = P(G^2) \Rightarrow Q(G^1) = Q(G^2)
\end{align}
for all graphs \(G^1\) and \(G^2\). If each of the graph polynomials \(P(G)\) and \(Q(G)\) has at least the distinctive power of the other, that means each pair of graphs is distinguished by both or none of them, then we say both graph polynomials have the same distinctive power, that is
\begin{align}
P(G^1) = P(G^2) \Leftrightarrow Q(G^1) = Q(G^2)
\end{align}
for all graphs \(G^1\) and \(G^2\).

Obviously, if we can calculate the value of a graph polynomial \(P(G)\) from the value of the graph polynomial \(Q(G)\), then \(P(G)\) has at least the the distinctive power of \(Q(G)\). If we additionally can calculate \(P(G)\) from \(Q(G)\), then both graph polynomials have the same distinctive power.

As we can calculate many graph polynomials from the value of the covered components polynomial \(C(G)\), this has at least the distinctive power of each of them.

To get an easier overview about the relations between the mentioned graph polynomials, we give a ``graph of graph polynomials'' in Figure \ref{figure:relations}. There we have a directed edge from graph polynomial \(P\) to graph polynomial \(Q\), if \(P\) has at least the distinctive power of \(Q\). Some edges following from transitivity are omitted.

\begin{figure}
\centering
\includegraphics{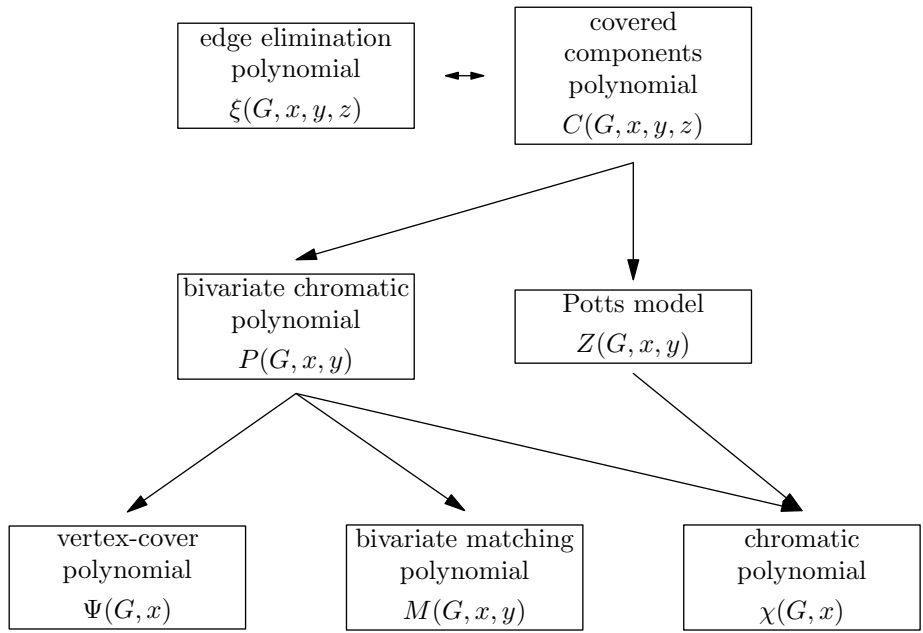}
\caption{Dependencies between graph polynomials with recurrence relations using the three edge operations.}
\label{figure:relations}
\end{figure}

In \cite[Question 2]{averbouch2008} the authors ask, if the edge elimination polynomial has more distinctive power than (at least but not the same distinctive power of) the bivariate chromatic polynomial and the Tutte polynomial together. That means if there is a pair of graphs \(G^1\) and \(G^2\) with same bivariate chromatic polynomial and same Tutte polynomial, but different edge elimination polynomial. In fact, there are such pairs.

\begin{rema}
\label{rema:distictive_power_eep_bcp}
The edge elimination polynomial \(\xi(G, x, y, z)\) and the covered components polynomial \(C(G, x, y, z)\) both have more distinctive power then the bivariate chromatic polynomial \(P(G, x, y)\) and the Tutte polynomial \(T(G, x, y)\) together. The ``minimum'' pair of non-isomorphic simple graphs showing this are the graphs \(G^1\) and \(G^2\) displayed in Figure \ref{fig:distinctive_power}.
\end{rema}
But it is not possible to find such graphs which are forests, because the bivariate chromatic polynomials has the same distinctive power for forests as the edge elimination polynomial; see Section \ref{sect:specialization_forests}.

We also give the minimum pair of non-isomorphic simple graphs and trees with the same edge elimination polynomial and covered components polynomial.

\begin{rema}
The edge elimination polynomial \(\xi(G, x, y, z)\) and the covered components polynomial \(C(G, x, y, z)\) both distinguish simple graphs with less than \(8\) vertices and with \(8\) vertices and less than \(12\) edges and they distinguish trees with less than \(10\) vertices. The ``minimum'' pairs of non-isomorphic simple graphs and trees, which can not be distinguished, are the graphs \(G^3\), \(G^4\) and \(G^5\), \(G^6\) of Figure \ref{fig:distinctive_power}.
\end{rema}
\begin{figure}
\begin{center}
\input{iso_graphs}
\end{center}
\caption{The graphs \(G^1\), \(G^2\), \(G^3\), \(G^4\), \(G^5\), and \(G^6\).}
\label{fig:distinctive_power}
\end{figure}

%% file: iso_graphs.tex
\begin{graph}(12,10)(-6,-4.75)
%\begin{framegraph}(12,10)(-6,-4.75)
	\graphnodecolour{1}
	\rectnode{g1}[0,0](-2.5,3.75)
	\autonodetext{g1}{
	\begin{graph}(4,3)(-2,-2.5)			
	%\begin{framegraph}(4,3)(-2,-2.5)
		\roundnode{1}(0,0)
		\roundnode{2}(0,-1)
		\roundnode{3}(1,-1)
		\roundnode{4}(-1,-1)	
		\roundnode{5}(0,-2)
		\roundnode{6}(1,-2)
		\edge{1}{2}	
		\edge{1}{3}
		\edge{2}{3}
		\edge{2}{4}
		\edge{2}{5}
		\edge{5}{6}
		\freetext(-1.5,-2){\(G_1\)}	
	\end{graph}}
	%\end{framegraph}}
	\rectnode{g2}[0,0](2.5,3.75)
	\autonodetext{g2}{
	\begin{graph}(4,3)(-2,-2.5)
	%\begin{framegraph}(4,3)(-2,-2.5)
		\roundnode{1}(0,0)
		\roundnode{2}(0,-1)
		\roundnode{3}(1,-1)
		\roundnode{4}(-1,0)	
		\roundnode{5}(0,-2)
		\roundnode{6}(1,-2)
		\edge{1}{2}	
		\edge{1}{3}
		\edge{1}{4}
		\edge{2}{5}
		\edge{3}{6}
		\edge{2}{3}
		\freetext(-1.5,-2){\(G_2\)}
	\end{graph}}
	%\end{framegraph}}
	\rectnode{g3}[0,0](-2.5,0)
	\autonodetext{g3}{
	\begin{graph}(4,4)(-2,-2)
	%\begin{framegraph}(4,4)(-2,-2)
		\roundnode{5}(0.75,0.75)
		\roundnode{0}(0,-1.5)
		\roundnode{2}(-1.5,0)
		\roundnode{3}(0,1.5)
		\roundnode{1}(0,0)	
		\roundnode{4}(0.75,-0.75)
		\roundnode{7}(-0.75,0.75)
		\roundnode{6}(-0.75,-0.75)
		\edge{0}{4}	
		\edge{0}{5}
		\edge{0}{6}
		\edge{1}{5}
		\edge{1}{6}
		\edge{1}{7}
		\edge{2}{6}
		\edge{2}{7}
		\edge{3}{7}
		\edge{4}{5}
		\edge{5}{7}
		\edge{6}{7}
		\freetext(-1.5,-1.5){\(G_{3}\)}
	\end{graph}}
	%\end{framegraph}}
	\rectnode{g4}[0,0](2.5,0)
	\autonodetext{g4}{
	\begin{graph}(4,4)(-2,-2)
	%\begin{framegraph}(4,4)(-2,-2)
		\roundnode{0}(0.75,0.75)
		\roundnode{1}(0,-1.5)
		\roundnode{2}(-1.5,0)
		\roundnode{3}(0,1.5)
		\roundnode{4}(0,0)	
		\roundnode{5}(0.75,-0.75)
		\roundnode{6}(-0.75,0.75)
		\roundnode{7}(-0.75,-0.75)
		\edge{0}{4}
		\edge{0}{5}
		\edge{0}{6}
		\edge{1}{5}
		\edge{1}{7}
		\edge{2}{6}
		\edge{2}{7}
		\edge{3}{6}
		\edge{4}{5}
		\edge{4}{7}
		\edge{5}{7}
		\edge{6}{7}
		\freetext(-1.5,-1.5){\(G_{4}\)}	
	\end{graph}}	
	%\end{framegraph}}
	\rectnode{g5}[0,0](-2.5,-3.5)
	\autonodetext{g5}{
	\begin{graph}(4,2.5)(-1.625,-1.875)
	%\begin{framegraph}(4,2.5)(-1.625,-1.875)
		\roundnode{1}(-1.50, 0)
		\roundnode{2}(-0.75, 0)
		\roundnode{3}( 0.00, 0)
		\roundnode{4}( 0.75, 0)
		\roundnode{5}( 1.50, 0)
		\roundnode{6}(-0.75,-0.75)
		\roundnode{7}( 0.00,-0.75)
		\roundnode{8}( 0.75,-0.75)
		\roundnode{9}( 1.50,-0.75)
		\roundnode{10}( 2.25,-0.75)
		\edge{1}{2}
		\edge{2}{3}
		\edge{2}{6}
		\edge{3}{4}
		\edge{3}{7}
		\edge{4}{5}
		\edge{7}{8}
		\edge{8}{9}
		\edge{9}{10}
		\freetext(-1.125,-1.5){\(G_5\)}
	\end{graph}}
	%\end{framegraph}}
	\rectnode{g6}[0,0](2.5,-3.5)
	\autonodetext{g6}{
	\begin{graph}(4,2.5)(-2,-1.875)
	%\begin{framegraph}(4,2.5)(-2,-1.875)
		\roundnode{1}(-1.50, 0)
		\roundnode{2}(-0.75, 0)
		\roundnode{3}( 0.00, 0)
		\roundnode{4}( 0.75, 0)
		\roundnode{5}( 1.50, 0)
		\roundnode{6}(-1.50,-0.75)
		\roundnode{7}(-0.75,-0.75)
		\roundnode{8}( 0.00,-0.75)
		\roundnode{9}( 0.75,-0.75)
		\roundnode{10}( 1.50,-0.75)
		\edge{1}{2}
		\edge{2}{3}
		\edge{2}{7}
		\edge{3}{4}
		\edge{4}{5}
		\edge{6}{7}
		\edge{7}{8}
		\edge{8}{9}
		\edge{9}{10}
		\freetext(-1.5,-1.5){\(G_6\)}
	\end{graph}}
	%\end{framegraph}}
\end{graph}
%\end{framegraph}

%% file: random_subgraphs.tex
\section{Random subgraphs}

\label{sect:random_subgraphs}

If we assume that the edges of a graph \(G = (V, E)\) are intact/appear independently with a given probability \(p\), it is natural to ask for the probabilities that the spanning subgraphs (spanned by the appearing/intact edges) have some properties. A multitude of these probabilities can be calculated from the covered components polynomial of the graph \(G\). (The covered components polynomial is also a generalization of the reliability polynomial.) As an example we show how to determine the probability, that no isolated vertex occurs in the spanning subgraph.

If we are just interested in the number of isolated vertices for the spanning subgraphs, we can consider the covered components polynomial with \(z = \frac{1}{x}\):
\begin{align}
C \left( G, x, y, \frac{1}{x} \right) &= \sum_{A \subseteq E}{x^{k(\esg{G}{A})} \ y^{\abs{A}} \ x^{-c(\esg{G}{A})}} \notag \\
&= \sum_{A \subseteq E}{x^{i(\esg{G}{A})} \ y^{\abs{A}}}.
\end{align}

By setting \(x = 0\) in \(C \left( G, x, y, \frac{1}{x} \right)\) we can extract the terms, corresponding to subgraphs without isolated vertices:
\begin{align}
& \sum_{\substack{A \subseteq E \\ i(\esg{G}{A}) = 0}}{x^{i(\esg{G}{A})} \ y^{\abs{A}}} =\left. C \left( G, x, y, \frac{1}{x} \right) \right|_{x=0}.
\end{align}

To weight the terms with the probability, that exactly the represented number of edges is intact (that means \(i\) edges in the case of the term \(y^{i}\)), we have to multiply all terms with \((1-p)^{m(G)}\) and substitute \(y = \frac{p}{1-p}\). Hence for \(P_{i=0}\), the probability that the number of isolated vertices is zero, we get
\begin{align}
& P_{i=0}(G, p) = \left. (1-p)^{m(G)} \cdot C \left(G, x, \frac{p}{1-p}, \frac{1}{x} \right) \right|_{x=0}.
\end{align}

%% file: specialization_forests.tex
\section{Specialization for forests}

\label{sect:specialization_forests}

In this section we state some results concerning the covered components polynomial of a forest. A forest \(F = (V, E)\) is a acyclic graph, that is a graph not including any cycle as subgraph. Recall that an edge \(e \in E\) a pendant edge, if at least one of its incident vertices has degree \(1\).

By defining a corresponding graph polynomial, we show that for forests \(F = (V, E)\) it is sufficient to use a recurrence relation for pendant edges \(e \in E\) with respect to the graph operations contraction and deletion of the edge \(e\).

\begin{defi}
\label{defi:peep}
Let \(F = (V, E), F^1, F^2\) be forests and \(e \in E\) a pendant edge of \(F\). The graph polynomial \(\hat{\xi}(F) = \hat{\xi}(F, x, a, b)\) is defined as
\begin{align}
& \hat{\xi}(F) = a \cdot \hat{\xi}(F_{/e}) + b \cdot \hat{\xi}(F_{\dagger e}), \label{eq:defi_peep_1} \\
& \hat{\xi}(F^1 \cdotcup F^2) = \hat{\xi}(F^1) \cdot \hat{\xi}(F^2) \label{eq:defi_peep_2}, \\
%\intertext{and the initial condition}
& \hat{\xi}(K_1) = x. \label{eq:defi_peep_3}
\end{align}
\end{defi}

For forests \(F\), the graph polynomial \(\hat{\xi}(F, x, a, b)\) is a substitution instance of the edge elimination polynomial and of the covered components polynomial.

\begin{lemm}
\label{lemm:conn_peep_xi_ccp}
Let \(F = (V, E)\) be a forest. Then
\begin{align}
& \hat{\xi}(G, x, a, b) = \xi(G, x, a - x, b), \label{eq:conn_peep_xi_ccp_1} \\
& \xi(G, x, y, z) = \hat{\xi}(G, x, x + y, z), \label{eq:conn_peep_xi_ccp_2} \\
& \hat{\xi}(G, x, a, b) = C \left(G, x, a - x, 1 + \frac{b}{x (a-x)} \right), \label{eq:conn_peep_xi_ccp_3} \\
& C(G, x, y, z) = \hat{\xi}(G, x, x + y, x y z - x y). \label{eq:conn_peep_xi_ccp_4}
\end{align}
\end{lemm}

\begin{proof}
The statements follow directly from the recurrence relation for a pendant edge \(e\). For the edge elimination polynomial \(\xi(G) = \xi(G, x, y, z)\) we get
\begin{align*}
\xi(G) 
&= \xi(G_{-e}) + y \cdot \xi(G_{/e}) + z \cdot \xi(G_{\dagger e}) \\
&= \xi(K_1 \cdotcup G_{/e}) + y \cdot \xi(G_{/e}) + z \cdot \xi(G_{\dagger e}) \\
&= (x + y) \cdot \xi(G_{/e}) + z \cdot \xi(G_{\dagger e})
\end{align*}
and for the covered components polynomial \(C(G) = C(G, x, y, z)\) we have by Corollary \ref{coro:ccp_pendant_edge}
\begin{align*}
C(G) &= (x + y) \cdot C(G_{/e}) + (x y z - x y) \cdot C(G_{\dagger e}). \qedhere
\end{align*}
\end{proof}

The degree sequence of a forest can be extracted from its \(\hat{\xi}\)-polynomial, even from the specialization of this for \(a = 1\). We denote the number of vertices of the graph \(G = (V, E)\) with degree \(i\) by \(\# \deg_{i}(G)\), that is
\begin{align}
\# \deg_{i}(G) = \abs{\{v \in V \mid \deg_{G}(v) = i\}}.
\end{align}

\begin{theo}
\label{theo:peep_degseq}
Let \(F = (V, E)\) be a forest. Then for \(i \geq 2\) we have
\begin{align}
\# \deg_{i}(F) = [b^{1} x^{k(F)+i-2}](\hat{\xi}(F, x, 1, b)) - [b^{1} x^{k(F)+i-1}](\hat{\xi}(F, x, 1, b)). \label{eq:theo_peep_degseq}
\end{align}
\end{theo}

\begin{proof}
To shorten the notation we use \(\hat{\xi}_1(F) = \hat{\xi}(F, x, 1, b)\). Then the statement has the form
\begin{align*}
\# \deg_{i}(F) = [b^{1} x^{k(F)+i-2}](\hat{\xi}_1(F)) - [b^{1} x^{k(F)+i-1}](\hat{\xi}_1(F)).
\end{align*}
According to \cite{knuth1992b}, for a statement \(s\) we define \([s]\) by
\begin{align*}
[s] = \begin{cases}
1 & \text{if } s \text{ is true,} \\
0 & \text{if } s \text{ is false.}
\end{cases}
\end{align*}

In case \(F\) has exactly \(j\) components, the coefficient \([b^{0} x^{j}](\hat{\xi}_1(F))\) corresponds to the value of the graph polynomial of \(F\) with all edges contracted and hence is \(1\), otherwise it is \(0\):
\begin{align*}
[b^{0} x^{j}](\hat{\xi}_1(F)) = [k(F) = j].
\end{align*}
If we delete the vertices of a pendant edge \(e = \{u, v\}\) with \(\deg_F(u) = 1\) and \(\deg_F(v) = i \geq 2\) for the forest \(F\), then the number of components increases by \(i-2\). It follows
\begin{align*}
[b^{0} x^{k(F)+i-2}](\hat{\xi}_1(F_{\dagger e})) = [\deg_F(v) = i].
\end{align*}

The proof is by induction with respect to the number of edges. As basic step we consider the case of a forest without any edges. Then \(F\) is isomorph to the edgeless graph \(E_{n(F)}\) and \(\hat{\xi}_1(E_{n(F)}) = x^{n(F)}\). Consequently the statement holds, because 
\begin{align*}
\# \deg_{i}(F) = [b^{1} x^{k(F)+i-2}](\hat{\xi}_1(F)) - [b^{1} x^{k(F)+i-1}](\hat{\xi}_1(F)) = 0
\end{align*}
for all \(i \geq 2\).

As induction hypothesis we assume that the statement holds for all forests with less then \(m > 0\) edges. Then for a forest with exactly \(m\) edges the forest has at least one edge and consequently at least one pendant edge \(e = \{u, v\}\) with \(\deg_F(u) = 1\), to which Equation \eqref{eq:defi_peep_1}, the recurrence relation for such an edge, can be applied. It follows:
\begin{align*}
& \eqspace [b^{1} x^{k(F)+i-2}](\hat{\xi}_1(F)) - [b^{1} x^{k(F)+i-1}](\hat{\xi}_1(F)) \\
&= [b^{1} x^{k(F)+i-2}](\hat{\xi}_1(F_{/e}) + b \cdot \hat{\xi}_1(F_{\dagger e})) \\
& \eqspace - [b^{1} x^{k(F)+i-1}](\hat{\xi}_1(F_{/e}) + b \cdot \hat{\xi}_1(F_{\dagger e})) \\
&= [b^{1} x^{k(F)+i-2}](\hat{\xi}_1(F_{/e})) + [b^{0} x^{k(F)+i-2}](\hat{\xi}_1(F_{\dagger e})) \\
& \eqspace - [b^{1} x^{k(F)+i-1}](\hat{\xi}_1(F_{/e})) - [b^{0} x^{k(F)+i-1}](\hat{\xi}_1(F_{\dagger e})) \\
&= \# \deg_{i}(F_{/e}) + [\deg_F(v) = i] - [\deg_F(v) = i+1] \\
&= \begin{cases}
\# \deg_{i}(F_{/e}) + 1 & \text{if } \deg_F(v) = i, \\
\# \deg_{i}(F_{/e}) - 1 & \text{if } \deg_F(v) = i+1, \\
\end{cases} \\
&= \# \deg_{i}(F). \qedhere
\end{align*}
\end{proof}

\begin{coro}
The degree sequence of a forest \(F\) is coded in the edge elimnation polynomial \(\xi(F, x, y, z)\) and in the covered components polynomial \(C(F, x, y, z)\).
\end{coro}

\begin{proof}
Combining Lemma \ref{lemm:conn_peep_xi_ccp} and Theorem \ref{theo:peep_degseq} the number of vertices with degree equals \(i\) for all \(i \geq 2\) is given. The numbers of vertices with degree equals zero (the number of isolated vertices) and one are given in Equation \eqref{eq:invariants_i} and Theorem \ref{theo:invariants_d1}.
\end{proof}

In Remark \ref{rema:distictive_power_eep_bcp} we have referred to the graphs \(G^1\) and \(G^2\) of Figure \ref{fig:distinctive_power}, which can be distinguished by the edge elimination polynomial (and hence also by the covered components polynomial), but not by the bivariate chromatic polynomial. Restricted to forests, there are no such graphs with this property, because for forests the edge elimination polynomial, the covered components polynomial and the bivariate chromatic polynomial have the same distinctive power.

\begin{theo}
Let \(F = (V, E)\) be a forest with the bivariate chromatic polynomial \(P(F, x, y)\) and the edge elimination polynomial \(\xi(F, x, y, z)\). Then \(\xi(F, x, y, z)\) can be calculated from \(P(F, x, y)\).
\end{theo}
\begin{proof}
We give a constructive proof which shows, how to calculate \(\xi(F, x, y, z)\) from \(P(F, x, y)\). Recall the definition of \(\xi(F, x, y, z)\) and a result for the \(P(F, x, y)\) stated in \cite[Proposition 1]{averbouch2008}:
\begin{align*}
& P(F, x, y) = \xi(F, x, -1, x-y).
\end{align*}
Hence, if we substitute \(y\) by \(x-z\) in \(P(F, x, y)\) we get
\begin{align*}
& P(F, x, x-z) = \xi(G, x, -1, z).
\end{align*}
Application of the recurrence relations for edges until no edge remains corresponds to a branching process, where we have three branches at each level (one for each subgraph we have to calculate, \(F_{-e}\), \(F_{/e}\) and \(F_{\dagger e}\)) and where the leaves are graphs only containing isolated vertices.

Each of the terms \(a_{i,k} x^{i} z^{k}\) of \(\xi(F, x, -1, z)\) occurs if and only if we end up with \(i\) isolated vertices and used \(k\) times the branch corresponding to \(F_{\dagger e}\).

Each of the \(k\) branchings corresponding to \(F_{\dagger e}\) ``deletes'' two vertices, while the branches of \(F_{-e}\) and \(F_{/e}\) deletes no and one vertex, respectively. Knowing the number of vertices of \(F\) (it is equal to the degree of \(x\) in \(P(F, x, y)\)), we can calculate \(j = n(F) - i - 2k\), that is the number of edge deletions. Roughly spoken, by this we can ``substitute'' \(-1\) by \(y\); the term \(a_{i,j} x^{i} z^{k}\) in \(\xi(F, x, -1, z)\) gives rise to the term \(a_{i,j} x^{i} y^{n(F)-i-2k} z^{k}\) in \(\xi(F, x, y, z)\).
Hence we can calculate \(\xi(F, x, y, z)\) for a forest \(F\) by
\begin{align*}
& \xi(F, x, y, z) = \sum_{i, k} \pcoef{x^{i} z^{k}}{P(F, x, x-z)} x^{i} y^{n(F)-i-2k} z^{k},
\end{align*}
where \(n(F) = \deg_x P(F, x, y)\).
\end{proof}

\begin{coro}
The bivariate chromatic polynomial, the edge elimination polynomial and the covered components polynomial have for forests the same distinctive power, that means each pair of forests is distinguished by all or none of them.
\end{coro}

%% file: generalization_hypergraphs.tex
\section{Generalization for hypergraphs}

\label{sect:generalization_hypergraphs}

\begin{defi}
A hypergraph is a pair \(G=(V, E)\), where \(V\) is a set of vertices and \(E\) is a multiset of hyperedges, where a hyperedge \(e \in E\) is a nonempty subset of \(V\) (\(\emptyset \neq e \subseteq V\)). 
\end{defi}

With this definition graphs are just hypergraphs, where each hyperedge has at most two incident vertices (is a subset of at most two vertices). It is easy to see, that the graph operations \(-e\), \(/e\) and \(\dagger e\) can be applied without problems for edges of hypergraphs.

\begin{defi}
\label{defi:c_hg}
Let \(G = (V, E)\) be a hpyergraph. The \emph{covered components polynomial} \(C(G, x, y, z)\) is defined as
\begin{align}
& C(G, x, y, z) = \sum_{A \subseteq E}{z^{k(\esg{G}{A})} \ y^{\abs{A}} \ z^{c(\esg{G}{A})}}. \label{eq:defi_c_hg}
\end{align}
\end{defi}

While it is obvious, that the definition of the covered components polynomial can be generalized to hypergraphs, it is not trivial that this generalization satisfies the same recurrence relation as the original graph polynomial.

\begin{theo}
\label{theo:c_rec_hg}
Let \(G = (V,E), G^1, G^2\) be hypergraphs and \(e \in E\) an edge of \(G\). Then the covered components polynomial \(C(G) = C(G, x, y, z)\) satisfies
\begin{align}
& C(G) = C(G_{-e}) + y \cdot C(G_{/e}) + (x y z - x y) \cdot C(G_{\dagger e}), \label{eq:c_rec1_hg}\\
& C(G^1 \cdotcup G^2) = C(G^1) \cdot C(G^2), \label{eq:c_rec2_hg} \\
& C(K_1) = x. \label{eq:c_rec3_hg}
\end{align}
\end{theo}

\begin{proof}
The proof is exactly the same as for Theorem \ref{theo:ccp_rec} in the case of graphs, we only have to substitute the terms ``graph'' and ``edge'' by ``hypergraph'' and ``hyperedge'', respectively.
\end{proof}

The relations to the most graph invariants and other graph polynomials are remaining if they are defined adequately. As an example we give a definition of the bivariate chromatic polynomial for hypergraphs.

\begin{defi}
\label{defi:p_hg}
Let \(G = (V, E)\) be a hypergraph and \(x, y\) nonnegative integers with \(x \geq y\). The \emph{bivariate chromatic polynomial} \(P(G, x, y)\) is defined as the number of mappings (colorings) \(\phi: V \rightarrow \{1, \ldots, x\}\), such that 
\begin{align}
\forall e = \{v_1, \ldots, v_k\} \in E: \phi(v_1) = \ldots = \phi(v_k) = c \Rightarrow c > y.
\end{align}
\end{defi}

For this generalization of the bivariate chromatic polynomial the recurrence relation given in \cite[Proposition 1]{averbouch2008} remains.

\begin{theo}
Let \(G = (V, E), G^1, G^2\) be hypergraphs and \(e \in E\) an edge of \(G\). Then the bivariate chromatic polynomial \(P(G) = P(G, x, y)\) satisfies
\begin{align}
& P(G) = P(G_{-e}) - P(G_{/e}) + (x - y) \cdot P(G_{\dagger e}), \label{eq:p_rec1_hg} \\
& P(G^1 \cdotcup G^2) = P(G^1) \cdot P(G^2) \label{eq:p_rec2_hg} \\
& P(K_1) = x. \label{eq:p_rec3_hg}
\end{align}
\end{theo}

Even though the proof of the recurrence relation of the bivariate chromatic polynomial in \cite[Proposition 1]{averbouch2008} is given only for simple graphs (it is given for edges which neither loops nor parallel edges), it can be used to prove the preceding theorem by just extending it to hyperedges.

\begin{proof}
In this proof by the term coloring we mean a coloring counted by the bivariate chromatic polynomial for \(x, y\) nonnegative integers with \(x \geq y\). Let \(e \in E\) be a hyperedge of \(G\).

If we start counting the colorings of \(G\) by 
\begin{align*}
P(G_{-e})
\end{align*}
we exactly count those coloring to much, where all incident vertices of \(e\) are colored with \(c \leq y\). We can subtract the number of coloring where all incident vertices of \(e\) are colored with the same color by
\begin{align*}
- P(G_{/e}),
\end{align*}
but this time we subtract too much, namely the colorings where all incident vertices of \(e\) are colored with \(c > y\). If a vertex is colored with a color \(c > y\), then there is no restriction for the coloring of the incident vertices. Hence, we can count the number of colorings of \(G\), where all incident vertices of \(e\) are colored with \(c > y\) (\(c \in \{y+1, \ldots, x\}\)) by
\begin{align*}
(x - y) \cdot P(G_{\dagger e}).
\end{align*}
Summing the three terms we have Equation \eqref{eq:p_rec1_hg}. For the Equations \eqref{eq:p_rec2_hg} and \eqref{eq:p_rec3_hg} we only have to consider, that each component can be colored independently and a single vertex can be colored by each of the \(x\) colors.
\end{proof}

\begin{rema}
Let \(G = (V, E)\) be a hypergraph with covered components polynomial \(C(G, x, y, z)\) and bivariate chromatic polynomial \(P(G, x, y)\). Then
\begin{align}
P(G, x, y) = C(G, x, -1, \frac{y}{x}).
\end{align}
\end{rema}

%% file: conclusions.tex
\section{Conclusions and open problems}

\label{sect:open_problems}

We have defined the covered components polynomial of a graph and proven that it is (just) another representation of the edge elimination polynomial, because they can be derived from each other. While both graph polynomials are strongly related, the covered components polynomial enables by its definition as sum over spanning subgraphs an easy deduction of a multitude of invariants of the graph.

One of the main questions related to every graph polynomial is, which graphs it distinguishes and which not. We state two open problems which point in this direction.

While any further graph invariant of a graph coded in its covered components polynomial may be useful to state the \(C\)-uniqueness for some more graph classes, we are particularly interested in ``degree invariants''.

\begin{ques}
Is the maximum degree \(\Delta(G)\) (or even more the degree sequence) of a simple graph coded in its covered component polynomial?
\end{ques}

For trees we have answered this affirmatively in Theorem \ref{theo:peep_degseq}. From the general answer it will also follow whether or not the \(M_2\)-index \cite{li2003} of a simple graph defined as
\begin{align}
& M_2(G) = \sum_{e = \{u, v\} \in E}{\deg u \cdot \deg v}
\end{align}
is coded in its covered components polynomial.

By complete enumeration we have determined the minimal pairs (with respect to the number of vertices) of non-isomorph trees which can not be distinguished by the covered components polynomial (or the edge elimination polynomial), see Figure \ref{fig:coincided_trees} with all such pairs with \(11\) and \(12\) vertices. (With up to \(9\) vertices there are no such pairs and with \(10\) vertices there is only one such pair, namely the graphs \(G^5\) and \(G^6\) displayed in Figure \ref{fig:distinctive_power}.) But what are the structural properties behind this?

\begin{figure}
\begin{center}
\input{iso_trees}
\end{center}
\caption{All pairs of non-isomorphic trees with \(11\) or \(12\) vertices with same covered components polynomial \(C(G, x, y, z)\).}
\label{fig:coincided_trees}
\end{figure}
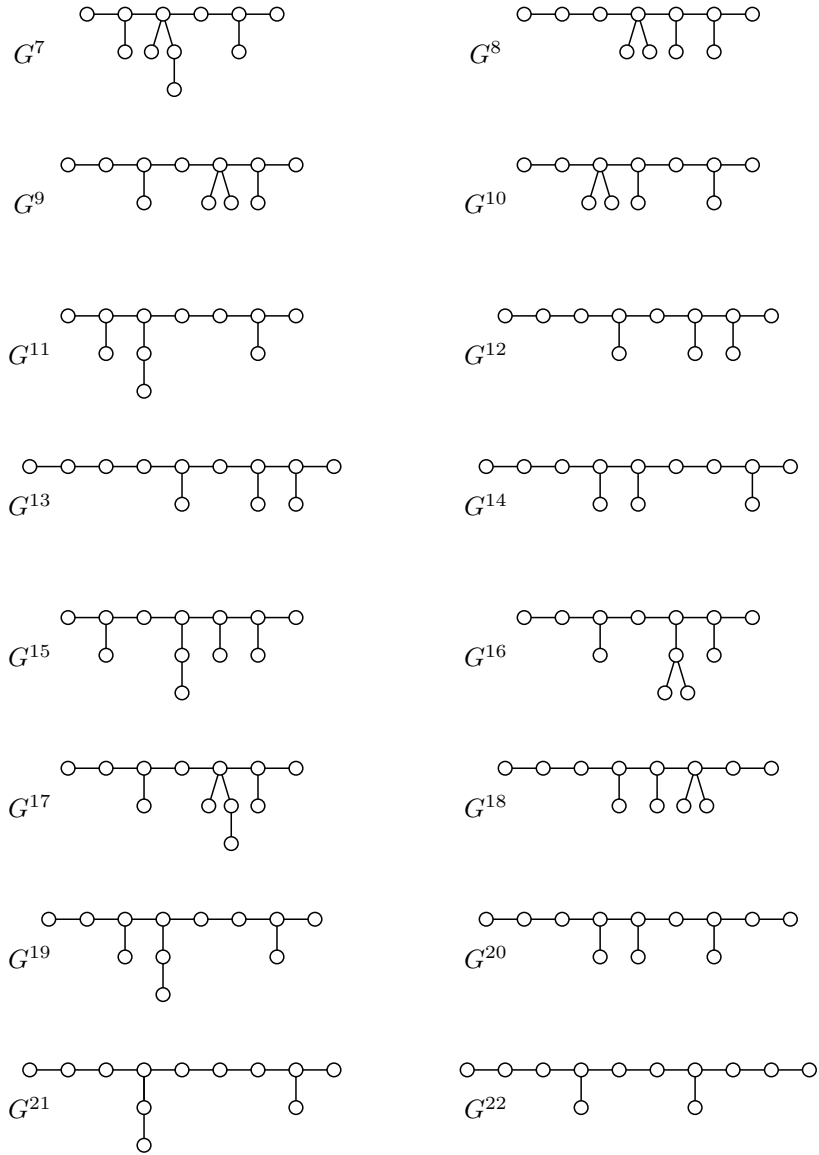

\begin{ques}
What are the structural properties of pairs of non-isomorph trees, which can not be distinguished by their covered components polynomial.
\end{ques}

%% file: iso_trees.tex
\begin{graph}(12,16)(-6,0)
%\begin{framegraph}(12,16)(-6,0)
	\graphnodecolour{1}
	\rectnode{g11-1a}[0,0](-3,15)
	\autonodetext{g11-1a}{
	\begin{graph}(6,2)(-3,-1.5)
	%\begin{framegraph}(6,2)(-3,-1.5)
		\roundnode{1}(-1.25,0)		
		\roundnode{2}(-0.75,0)		
		\roundnode{3}(-0.25,0)		
		\roundnode{4}(0.25,0)
		\roundnode{5}(0.75,0)
		\roundnode{6}(1.25,0)
		\roundnode{7}(-0.75,-0.50)	
		\roundnode{8}(-0.4,-0.5)
		\roundnode{9}(-0.1,-0.5)
		\roundnode{10}(-0.1,-1)
		\roundnode{11}(0.75,-0.5)
		\edge{1}{2}	
		\edge{2}{3}
		\edge{3}{4}
		\edge{4}{5}
		\edge{5}{6}
		\edge{2}{7}
		\edge{3}{8}
		\edge{3}{9}
		\edge{9}{10}
		\edge{5}{11}
		\freetext(-2,-0.5){\(G^{7}\)}	
	\end{graph}}
	%\end{framegraph}}
	\rectnode{g11-1b}[0,0](3,15)
	\autonodetext{g11-1b}{
	\begin{graph}(6,2)(-3,-1.5)			
	%\begin{framegraph}(6,2)(-3,-1.5)
		\roundnode{1}(-1.5,0)		
		\roundnode{2}(-1,0)		
		\roundnode{3}(-0.5,0)		
		\roundnode{4}(0,0)
		\roundnode{5}(0.5,0)
		\roundnode{6}(1,0)
		\roundnode{7}(1.5,0)	
		\roundnode{8}(-0.15,-0.5)
		\roundnode{9}(0.15,-0.5)
		\roundnode{10}(0.5,-0.5)
		\roundnode{11}(1,-0.5)
		\edge{1}{2}	
		\edge{2}{3}
		\edge{3}{4}
		\edge{4}{5}
		\edge{5}{6}
		\edge{6}{7}
		\edge{4}{8}
		\edge{4}{9}
		\edge{5}{10}
		\edge{6}{11}
		\freetext(-2,-0.5){\(G^{8}\)}
	\end{graph}}
	%\end{framegraph}}
	\rectnode{g11-2a}[0,0](-3,13)
	\autonodetext{g11-2a}{
	\begin{graph}(6,2)(-3,-1.5)
	%\begin{framegraph}(6,2)(-3,-1.5)
		\roundnode{1}(-1.5,0)		
		\roundnode{2}(-1,0)		
		\roundnode{3}(-0.5,0)		
		\roundnode{4}(0,0)
		\roundnode{5}(0.5,0)
		\roundnode{6}(1,0)
		\roundnode{7}(1.5,0)	
		\roundnode{8}(-0.5,-0.5)
		\roundnode{9}(0.35,-0.5)
		\roundnode{10}(0.65,-0.5)
		\roundnode{11}(1,-0.5)
		\edge{1}{2}	
		\edge{2}{3}
		\edge{3}{4}
		\edge{4}{5}
		\edge{5}{6}
		\edge{6}{7}
		\edge{3}{8}
		\edge{5}{9}
		\edge{5}{10}
		\edge{6}{11}
		\freetext(-2,-0.5){\(G^{9}\)}	
	\end{graph}}
	%\end{framegraph}}
	\rectnode{g11-2b}[0,0](3,13)
	\autonodetext{g11-2b}{
	\begin{graph}(6,2)(-3,-1.5)
	%\begin{framegraph}(6,2)(-3,-1.5)
		\roundnode{1}(-1.5,0)		
		\roundnode{2}(-1,0)		
		\roundnode{3}(-0.5,0)		
		\roundnode{4}(0,0)
		\roundnode{5}(0.5,0)
		\roundnode{6}(1,0)
		\roundnode{7}(1.5,0)	
		\roundnode{8}(-0.35,-0.5)
		\roundnode{9}(-0.65,-0.5)
		\roundnode{10}(0,-0.5)
		\roundnode{11}(1,-0.5)
		\edge{1}{2}	
		\edge{2}{3}
		\edge{3}{4}
		\edge{4}{5}
		\edge{5}{6}
		\edge{6}{7}
		\edge{3}{8}
		\edge{3}{9}
		\edge{4}{10}
		\edge{6}{11}
		\freetext(-2,-0.5){\(G^{10}\)}
	\end{graph}}
	%\end{framegraph}}	
	\rectnode{g11-3a}[0,0](-3,11)
	\autonodetext{g11-3a}{
	\begin{graph}(6,2)(-3,-1.5)
	%\begin{framegraph}(6,2)(-3,-1.5)
		\roundnode{1}(-1.5,0)		
		\roundnode{2}(-1,0)		
		\roundnode{3}(-0.5,0)		
		\roundnode{4}(0,0)
		\roundnode{5}(0.5,0)
		\roundnode{6}(1,0)
		\roundnode{7}(1.5,0)	
		\roundnode{8}(-1,-0.5)
		\roundnode{9}(-0.5,-0.5)
		\roundnode{10}(-0.5,-1)
		\roundnode{11}(1,-0.5)
		\edge{1}{2}	
		\edge{2}{3}
		\edge{3}{4}
		\edge{4}{5}
		\edge{5}{6}
		\edge{6}{7}
		\edge{2}{8}
		\edge{3}{9}
		\edge{9}{10}
		\edge{6}{11}
		\freetext(-2,-0.5){\(G^{11}\)}	
	\end{graph}}
	%\end{framegraph}}
	\rectnode{g11-3b}[0,0](3,11)
	\autonodetext{g11-3b}{
	\begin{graph}(6,2)(-3,-1.5)
	%\begin{framegraph}(6,2)(-3,-1.5)
		\roundnode{1}(-1.75,0)		
		\roundnode{2}(-1.25,0)		
		\roundnode{3}(-0.75,0)		
		\roundnode{4}(-0.25,0)
		\roundnode{5}(0.25,0)
		\roundnode{6}(0.75,0)
		\roundnode{7}(1.25,0)	
		\roundnode{8}(1.75,0)
		\roundnode{9}(-0.25,-0.5)
		\roundnode{10}(0.75,-0.5)
		\roundnode{11}(1.25,-0.5)
		\edge{1}{2}	
		\edge{2}{3}
		\edge{3}{4}
		\edge{4}{5}
		\edge{5}{6}
		\edge{6}{7}
		\edge{7}{8}
		\edge{4}{9}
		\edge{6}{10}
		\edge{7}{11}
		\freetext(-2,-0.5){\(G^{12}\)}	
	\end{graph}}
	%\end{framegraph}}	
	\rectnode{g12-1a}[0,0](-3,9)
	\autonodetext{g12-1a}{
	\begin{graph}(6,2)(-3,-1.5)
	%\begin{framegraph}(6,2)(-3,-1.5)
		\roundnode{1}(-2,0)		
		\roundnode{2}(-1.5,0)		
		\roundnode{3}(-1,0)		
		\roundnode{4}(-0.5,0)		
		\roundnode{5}(0,0)
		\roundnode{6}(0.5,0)
		\roundnode{7}(1,0)
		\roundnode{8}(1.5,0)	
		\roundnode{9}(2,0)
		\roundnode{10}(0,-0.5)
		\roundnode{11}(1,-0.5)
		\roundnode{12}(1.5,-0.5)
		\edge{1}{2}	
		\edge{2}{3}
		\edge{3}{4}
		\edge{4}{5}
		\edge{5}{6}
		\edge{6}{7}
		\edge{7}{8}
		\edge{8}{9}
		\edge{5}{10}
		\edge{7}{11}
		\edge{8}{12}
		\freetext(-2,-0.5){\(G^{13}\)}	
	\end{graph}}
	%\end{framegraph}}
	\rectnode{g12-1b}[0,0](3,9)
	\autonodetext{g12-1b}{
	\begin{graph}(6,2)(-3,-1.5)
	%\begin{framegraph}(6,2)(-3,-1.5)
		\roundnode{1}(-2,0)		
		\roundnode{2}(-1.5,0)		
		\roundnode{3}(-1,0)		
		\roundnode{4}(-0.5,0)		
		\roundnode{5}(0,0)
		\roundnode{6}(0.5,0)
		\roundnode{7}(1,0)
		\roundnode{8}(1.5,0)	
		\roundnode{9}(2,0)
		\roundnode{10}(-0.5,-0.5)
		\roundnode{11}(0,-0.5)
		\roundnode{12}(1.5,-0.5)
		\edge{1}{2}	
		\edge{2}{3}
		\edge{3}{4}
		\edge{4}{5}
		\edge{5}{6}
		\edge{6}{7}
		\edge{7}{8}
		\edge{8}{9}
		\edge{4}{10}
		\edge{5}{11}
		\edge{8}{12}
		\freetext(-2,-0.5){\(G^{14}\)}	
	\end{graph}}
	%\end{framegraph}}
	\rectnode{g12-2a}[0,0](-3,7)
	\autonodetext{g12-2a}{
	\begin{graph}(6,2)(-3,-1.5)
	%\begin{framegraph}(6,2)(-3,-1.5)
		\roundnode{1}(-1.5,0)		
		\roundnode{2}(-1,0)		
		\roundnode{3}(-0.5,0)		
		\roundnode{4}(0,0)
		\roundnode{5}(0.5,0)
		\roundnode{6}(1,0)
		\roundnode{7}(1.5,0)	
		\roundnode{8}(-1,-0.5)
		\roundnode{9}(0,-0.5)
		\roundnode{10}(0,-1)
		\roundnode{11}(0.5,-0.5)
		\roundnode{12}(1,-0.5)
		\edge{1}{2}	
		\edge{2}{3}
		\edge{3}{4}
		\edge{4}{5}
		\edge{5}{6}
		\edge{6}{7}
		\edge{2}{8}
		\edge{4}{9}
		\edge{9}{10}
		\edge{5}{11}
		\edge{6}{12}
		\freetext(-2,-0.5){\(G^{15}\)}	
	\end{graph}}
	%\end{framegraph}}
	\rectnode{g12-2b}[0,0](3,7)
	\autonodetext{g12-2b}{
	\begin{graph}(6,2)(-3,-1.5)
	%\begin{framegraph}(6,2)(-3,-1.5)
		\roundnode{1}(-1.5,0)		
		\roundnode{2}(-1,0)		
		\roundnode{3}(-0.5,0)		
		\roundnode{4}(0,0)
		\roundnode{5}(0.5,0)
		\roundnode{6}(1,0)
		\roundnode{7}(1.5,0)	
		\roundnode{8}(-0.5,-0.5)
		\roundnode{9}(0.5,-0.5)
		\roundnode{10}(0.35,-1)
		\roundnode{11}(0.65,-1)
		\roundnode{12}(1,-0.5)
		\edge{1}{2}	
		\edge{2}{3}
		\edge{3}{4}
		\edge{4}{5}
		\edge{5}{6}
		\edge{6}{7}
		\edge{3}{8}
		\edge{5}{9}
		\edge{9}{10}
		\edge{9}{11}
		\edge{6}{12}
		\freetext(-2,-0.5){\(G^{16}\)}	
	\end{graph}}
	%\end{framegraph}}
	\rectnode{g12-3a}[0,0](-3,5)
	\autonodetext{g12-3a}{
	\begin{graph}(6,2)(-3,-1.5)
	%\begin{framegraph}(6,2)(-3,-1.5)
		\roundnode{1}(-1.5,0)		
		\roundnode{2}(-1,0)		
		\roundnode{3}(-0.5,0)		
		\roundnode{4}(0,0)
		\roundnode{5}(0.5,0)
		\roundnode{6}(1,0)
		\roundnode{7}(1.5,0)	
		\roundnode{8}(-0.5,-0.5)
		\roundnode{9}(0.35,-0.5)
		\roundnode{10}(0.65,-0.5)
		\roundnode{11}(0.65,-1)
		\roundnode{12}(1,-0.5)
		\edge{1}{2}	
		\edge{2}{3}
		\edge{3}{4}
		\edge{4}{5}
		\edge{5}{6}
		\edge{6}{7}
		\edge{3}{8}
		\edge{5}{9}
		\edge{5}{10}
		\edge{10}{11}
		\edge{6}{12}
		\freetext(-2,-0.5){\(G^{17}\)}	
	\end{graph}}
	%\end{framegraph}}
	\rectnode{g12-3b}[0,0](3,5)
	\autonodetext{g12-3b}{
	\begin{graph}(6,2)(-3,-1.5)
	%\begin{framegraph}(6,2)(-3,-1.5)
		\roundnode{1}(-1.75,0)		
		\roundnode{2}(-1.25,0)		
		\roundnode{3}(-0.75,0)		
		\roundnode{4}(-0.25,0)
		\roundnode{5}(0.25,0)
		\roundnode{6}(0.75,0)
		\roundnode{7}(1.25,0)	
		\roundnode{8}(1.75,0)
		\roundnode{9}(-0.25,-0.5)
		\roundnode{10}(0.25,-0.5)
		\roundnode{11}(0.6,-0.5)
		\roundnode{12}(0.9,-0.5)
		\edge{1}{2}	
		\edge{2}{3}
		\edge{3}{4}
		\edge{4}{5}
		\edge{5}{6}
		\edge{6}{7}
		\edge{7}{8}
		\edge{4}{9}
		\edge{5}{10}
		\edge{6}{11}
		\edge{6}{12}
		\freetext(-2,-0.5){\(G^{18}\)}	
	\end{graph}}
	%\end{framegraph}}
	\rectnode{g12-4a}[0,0](-3,3)
	\autonodetext{g12-4a}{
	\begin{graph}(6,2)(-3,-1.5)
	%\begin{framegraph}(6,2)(-3,-1.5)
	\roundnode{1}(-1.75,0)		
		\roundnode{2}(-1.25,0)		
		\roundnode{3}(-0.75,0)		
		\roundnode{4}(-0.25,0)
		\roundnode{5}(0.25,0)
		\roundnode{6}(0.75,0)
		\roundnode{7}(1.25,0)	
		\roundnode{8}(1.75,0)
		\roundnode{9}(-0.75,-0.5)
		\roundnode{10}(-0.25,-0.5)
		\roundnode{11}(-0.25,-1)
		\roundnode{12}(1.25,-0.5)
		\edge{1}{2}	
		\edge{2}{3}
		\edge{3}{4}
		\edge{4}{5}
		\edge{5}{6}
		\edge{6}{7}
		\edge{7}{8}
		\edge{3}{9}
		\edge{4}{10}
		\edge{10}{11}
		\edge{7}{12}
		\freetext(-2,-0.5){\(G^{19}\)}	
	\end{graph}}
	%\end{framegraph}}
	\rectnode{g12-4b}[0,0](3,3)
	\autonodetext{g12-4b}{
	\begin{graph}(6,2)(-3,-1.5)
	%\begin{framegraph}(6,2)(-3,-1.5)
		\roundnode{1}(-2,0)		
		\roundnode{2}(-1.5,0)		
		\roundnode{3}(-1,0)		
		\roundnode{4}(-0.5,0)		
		\roundnode{5}(0,0)
		\roundnode{6}(0.5,0)
		\roundnode{7}(1,0)
		\roundnode{8}(1.5,0)	
		\roundnode{9}(2,0)
		\roundnode{10}(-0.5,-0.5)
		\roundnode{11}(0,-0.5)
		\roundnode{12}(1,-0.5)
		\edge{1}{2}	
		\edge{2}{3}
		\edge{3}{4}
		\edge{4}{5}
		\edge{5}{6}
		\edge{6}{7}
		\edge{7}{8}
		\edge{8}{9}
		\edge{4}{10}
		\edge{5}{11}
		\edge{7}{12}
		\freetext(-2,-0.5){\(G^{20}\)}	
	\end{graph}}
	%\end{framegraph}}
	\rectnode{g12-5a}[0,0](-3,1)
	\autonodetext{g12-5a}{
	\begin{graph}(6,2)(-3,-1.5)
	%\begin{framegraph}(6,2)(-3,-1.5)
		\roundnode{1}(-2,0)		
		\roundnode{2}(-1.5,0)		
		\roundnode{3}(-1,0)		
		\roundnode{4}(-0.5,0)		
		\roundnode{5}(0,0)
		\roundnode{6}(0.5,0)
		\roundnode{7}(1,0)
		\roundnode{8}(1.5,0)	
		\roundnode{9}(2,0)
		\roundnode{10}(-0.5,-0.5)
		\roundnode{11}(-0.5,-1)
		\roundnode{12}(1.5,-0.5)
		\edge{1}{2}	
		\edge{2}{3}
		\edge{3}{4}
		\edge{4}{5}
		\edge{5}{6}
		\edge{6}{7}
		\edge{7}{8}
		\edge{8}{9}
		\edge{4}{10}
		\edge{4}{11}
		\edge{8}{12}
		\freetext(-2,-0.5){\(G^{21}\)}	
	\end{graph}}
	%\end{framegraph}}
	\rectnode{g12-5b}[0,0](3,1)
	\autonodetext{g12-5b}{
	\begin{graph}(6,2)(-3,-1.5)	
	%\begin{framegraph}(6,2)(-3,-1.5)
		\roundnode{1}(-2.25,0)		
		\roundnode{2}(-1.75,0)		
		\roundnode{3}(-1.25,0)		
		\roundnode{4}(-0.75,0)		
		\roundnode{5}(-0.25,0)
		\roundnode{6}(0.25,0)
		\roundnode{7}(0.75,0)
		\roundnode{8}(1.25,0)	
		\roundnode{9}(1.75,0)
		\roundnode{10}(2.25,0)
		\roundnode{11}(-0.75,-0.5)
		\roundnode{12}(0.75,-0.5)
		\edge{1}{2}	
		\edge{2}{3}
		\edge{3}{4}
		\edge{4}{5}
		\edge{5}{6}
		\edge{6}{7}
		\edge{7}{8}
		\edge{8}{9}
		\edge{9}{10}
		\edge{4}{11}
		\edge{7}{12}
		\freetext(-2,-0.5){\(G^{22}\)}	
	\end{graph}}
	%\end{framegraph}}
\end{graph}
%\end{framegraph}

%% file: paper.bbl
\begin{thebibliography}{10}

\bibitem{averbouch2010c}
Ilia Averbouch.
\newblock {\em Completeness and Universality Properties of Graph Invariants and
  Graph Polynomials}.
\newblock PhD thesis, Israel Institute of Technology, November 2010.
\newblock URL: \url{http://cs.technion.ac.il/~ailia/thesis/Work/thesis.pdf}.

\bibitem{averbouch2008}
Ilia Averbouch, Benny Godlin, and Johann~A. Makowsky.
\newblock A most general edge elimination polynomial.
\newblock In {\em Graph-Theoretic Concepts in Computer Science}, volume 5344 of
  {\em Lecture Notes in Computer Science}, pages 31 -- 42. Springer, Berlin /
  Heidelberg, 2008.
\newblock \href {http://dx.doi.org/10.1007/978-3-540-92248-3_4}
  {\path{doi:10.1007/978-3-540-92248-3_4}}.

\bibitem{averbouch2010}
Ilia Averbouch, Benny Godlin, and Johann~A. Makowsky.
\newblock An extension of the bivariante chromatic polynomial.
\newblock {\em European Journal of Combinatorics}, 31(1):1 -- 17, 2010.
\newblock \href {http://dx.doi.org/10.1016/j.ejc.2009.05.006}
  {\path{doi:10.1016/j.ejc.2009.05.006}}.

\bibitem{birkhoff1912}
George~D. Birkhoff.
\newblock A determinant formula for the number of ways of coloring a map.
\newblock {\em The Annals of Mathematics}, 14(1):42 -- 46, 1912.
\newblock URL: \url{http://www.jstor.org/stable/1967597}.

\bibitem{mier2004}
Anna de~Mier and Marc Noy.
\newblock On graphs determined by their {T}utte polynomial.
\newblock {\em Graphs and Combinatorics}, 20(1):105--119, 2004.
\newblock \href {http://dx.doi.org/10.1007/s00373-003-0534-z}
  {\path{doi:10.1007/s00373-003-0534-z}}.

\bibitem{diestel2005}
Reinhard Diestel.
\newblock {\em Graph Theory}, volume 173 of {\em Graduate Texts in
  Mathematics}.
\newblock Springer, Berlin/Heidelberg/New York, 3rd edition, 2005.
\newblock URL: \url{http://diestel-graph-theory.com}.

\bibitem{dohmen2003}
Klaus Dohmen, André Pönitz, and Peter Tittmann.
\newblock A new two-variable generalization of the chromatic polynomial.
\newblock {\em Discrete Mathematics and Theoretical Computer Science}, 6:69 --
  90, 2003.
\newblock URL:
  \url{http://www.emis.de/journals/DMTCS/volumes/abstracts/pdfpapers/dm060106.pdf}.

\bibitem{dong2002}
F.~M. Dong, M.~D. Hendy, K.~L. Teo, and C.~H.~C. Little.
\newblock The vertex-cover polynomial of a graph.
\newblock {\em Discrete Mathematics}, 250(1 - 3):71 -- 78, 2002.
\newblock \href {http://dx.doi.org/10.1016/S0012-365X(01)00272-2}
  {\path{doi:10.1016/S0012-365X(01)00272-2}}.

\bibitem{dong2005}
F.~M. Dong, K.~M. Koh, and K.~L. Teo.
\newblock {\em Chromatic polynomials and chromaticity of graphs}.
\newblock World Scientific Publishing, 2005.

\bibitem{hoede1994}
Cornelis Hoede and Xueliang Li.
\newblock Clique polynomials and independent set polynomials of graphs.
\newblock {\em Discrete Mathematics}, 125:219 -- 228, 1994.
\newblock \href {http://dx.doi.org/10.1016/0012-365X(94)90163-5}
  {\path{doi:10.1016/0012-365X(94)90163-5}}.

\bibitem{hoffmann2010}
Christian Hoffmann.
\newblock {\em Computational Complexity of Graph Polynomials}.
\newblock PhD thesis, Saarland University, 2010.
\newblock URL:
  \url{http://research.enn-hoch.de/dissertation_christian_hoffmann_pflichtexemplar.pdf}.

\bibitem{hoffmann2010b}
Christian Hoffmann.
\newblock A most general edge elimination polynomial - thickening of edges.
\newblock {\em Fundamenta Informaticae}, 98(4):373 -- 378, 2010.
\newblock \href {http://dx.doi.org/10.3233/FI-2010-233}
  {\path{doi:10.3233/FI-2010-233}}.

\bibitem{knuth1992b}
Donald~Ervin Knuth.
\newblock Two notes on notation.
\newblock {\em The American Mathematical Monthly}, 99(5):403 -- 422, 1992.
\newblock URL: \url{http://www.jstor.org/stable/2325085}.

\bibitem{koh1990}
K.~M. Koh and K.~L. Teo.
\newblock The search for chromatically unique graphs.
\newblock {\em Graphs and Combinatorics}, 6(3):259 -- 285, 1990.
\newblock \href {http://dx.doi.org/10.1007/BF01787578}
  {\path{doi:10.1007/BF01787578}}.

\bibitem{levit2005}
Vadim~E. Levit and Eugen Madrescu.
\newblock The independence polynomial of a graph - a survey.
\newblock In {\em Proceedings of the 1st International Conference on Algebraic
  Informatics}, pages 233--254, 2005.
\newblock URL: \url{http://web.auth.gr/cai05/papers/20.pdf}.

\bibitem{li2003}
Xueliang Li, Zimao Li, and Lusheng Wang.
\newblock The inverse problems for some topological indices in combinatorial
  chemistry.
\newblock {\em Journal of Computational Biology}, 10(1):47 -- 55, 2003.
\newblock \href {http://dx.doi.org/10.1089/106652703763255660}
  {\path{doi:10.1089/106652703763255660}}.

\bibitem{makowsky2008}
Johann~A. Makowsky.
\newblock From a zoo to a zoology: {T}owards a general theory of graph
  polynomials.
\newblock {\em Theory of Computing Systems}, 43(3-4):542 -- 562, 2008.
\newblock \href {http://dx.doi.org/10.1007/s00224-007-9022-9}
  {\path{doi:10.1007/s00224-007-9022-9}}.

\bibitem{merrifield1981}
Richard~E. Merrifield and Howard~E. Simmons.
\newblock Enumeration of structure-sensitive graphical subsets: {T}heory.
\newblock {\em Proceedings of the National Acadamy of Sciences of the USA},
  78(2):692 -- 695, February 1981.
\newblock URL: \url{http://www.pnas.org/content/78/2/692.full.pdf+html}.

\bibitem{noy2003}
Marc Noy.
\newblock Graphs determined by polynomial invariants.
\newblock {\em Theoretical Computer Science}, 307(2):365 -- 384, 2003.
\newblock Random Generation of Combinatorial Objects and Bijective
  Combinatorics.
\newblock \href {http://dx.doi.org/10.1016/S0304-3975(03)00225-1}
  {\path{doi:10.1016/S0304-3975(03)00225-1}}.

\bibitem{simon2009}
Frank Simon.
\newblock Enumerative combinatorics in the partition lattice.
\newblock Master's thesis, University of Applied Sciences Mittweida, 2009.

\bibitem{sokal2005}
Alan~D. Sokal.
\newblock The multivariate {T}utte polynomial (alias {P}otts model) for graphs
  and matroids.
\newblock In {\em Surveys in combinatorics}, volume 327, pages 173--226, 2005.
\newblock \href {http://arxiv.org/abs/math/0503607}
  {\path{arXiv:math/0503607}}.

\bibitem{stanley1986}
Richard~Peter Stanley.
\newblock {\em Enumerative Combinatorics}, volume~1.
\newblock Wadsworth \& Brooks/Cole, 1986.

\bibitem{tittmann2011}
Peter Tittmann, Ilia Averbouch, and Johann~A. Makowsky.
\newblock The enumeration of vertex induced subgraphs with respect to the
  number of components.
\newblock {\em European Journal of Combinatorics}, 32(7):954 -- 974, 2011.
\newblock \href {http://dx.doi.org/10.1016/j.ejc.2011.03.017}
  {\path{doi:10.1016/j.ejc.2011.03.017}}.

\bibitem{white2011}
Jaboc~A. White.
\newblock On multivariate chromatic polynomials of hypergraphs and hyperedge
  elimination.
\newblock {\em The Electronic Journal of Combinatorics}, 18(1):\#P160, 2011.
\newblock URL:
  \url{http://www.combinatorics.org/Volume_18/Abstracts/v18i1p160.html}.

\end{thebibliography}
